\documentclass[]{imsart}

\RequirePackage[OT1]{fontenc}
\RequirePackage[square]{natbib}
\RequirePackage{amsthm,amsmath}
\RequirePackage[colorlinks,citecolor=blue,urlcolor=blue]{hyperref}

\arxiv{arXiv:0000.0000}

\startlocaldefs
\numberwithin{equation}{section}
\theoremstyle{plain}
\newtheorem{thm}{Theorem}[section]
\newtheorem{lemma}{Lemma}[section]
\newtheorem{corollary}{Corollary}[section]
\newtheorem{proposition}{Proposition}[section]

\theoremstyle{remark}

\endlocaldefs

\begin{document}

\begin{frontmatter}
\title{Self-normalized Cram\'{e}r type moderate deviations for martingales}
\runtitle{Self-normalized Cram\'{e}r type moderate deviations}

\begin{aug}
\author{\fnms{Xiequan} \snm{Fan}\thanksref{a,e1}\ead[label=e1,mark]{fanxiequan@hotmail.com}}
\author{\fnms{Ion} \snm{Grama}\thanksref{b,e2}\ead[label=e2,mark]{ion.grama@univ-ubs.fr}}
\author{\fnms{Quansheng} \snm{Liu}\thanksref{b,e3}\ead[label=e3,mark]{quansheng.liu@univ-ubs.fr}}
\and
\author{\fnms{Qi-Man} \snm{Shao}\thanksref{c,e4}\ead[label=e4,mark]{qmshao@sta.cuhk.edu.hk}}

\address[a]{Center for Applied Mathematics,
Tianjin University, Tianjin 300072,  China.
\printead{e1}}

\address[b]{Universit\'{e} de Bretagne-Sud, LMBA, UMR CNRS 6205,
 Campus de Tohannic, 56017 Vannes, France.
\printead{e2,e3}}

\address[c]{Department of Statistics, The Chinese University of Hong Kong, Shatin, NT, Hong Kong.
\printead{e4}}

\runauthor{X. Fan et al.}

\affiliation{Tianjin University, Universit\'{e} de Bretagne-Sud and The Chinese University of Hong Kong}

\end{aug}

\begin{abstract}
Let $(X _i,\mathcal{F}_i)_{i\geq1}$ be a  sequence of martingale differences.
Set $S_n=\sum_{i=1}^n X_i $ and $[ S]_n=\sum_{i=1}^n X_i^2.$
We prove a Cram\'er type moderate deviation expansion for $\mathbf{P}(S_n/\sqrt{[ S]_n} \geq x)$ as $n\to+\infty.$ Our results partly extend  the earlier work of \citep{JSW03}  for independent random variables.
\end{abstract}

\begin{keyword}
\kwd{Martingales}
\kwd{self-normalized sequences}
\kwd{Cram\'{e}r's moderate deviations}
\end{keyword}
\end{frontmatter}

\section{Introduction}
Let $(X_i)_{i\geq1}$ be a sequence of independent random variables with zero means and finite variances: $\mathbf{E}X_i=0$ and
$0< \mathbf{E}X_i^2< \infty$ for all $i\geq 1.$
Set
$$S_n=\sum_{i=1}^n X_i,\ \ \ \ \ B_n^2 = \sum_{i=1}^n \mathbf{E}X_i^2,  \ \  \ \ \  V_n^2= \sum_{i=1}^n X_i^2.$$
It is well-known that under the Lindeberg condition the central limit theorem (CLT) holds
$$\sup_{x\in \mathbf{R}}  \Big|\mathbf{P}(S_n/B_n \leq x ) -\Phi(x) \Big | \rightarrow 0\ \ \   \textrm{as}\ n\rightarrow \infty,$$
where $\Phi(x)$ denotes the standard normal distribution function.
Cram\'er's moderate deviation expansion stated below gives an estimation of the relative error
of $\mathbf{P}(S_n/B_n \geq x ) $ to $1-\Phi(x).$
If $(X_i)_{i\geq1}$ are identically distributed with $\mathbf{E}e^{t_0 \sqrt{|X_1|}} < \infty$
for some $t_0 >0,$  then for all $0\leq x =o(n^{1/6})$ as $n\rightarrow \infty,$
\begin{equation}\label{Cramer00}
\frac{\mathbf{P}(S_n/ B_n \geq x)}{1-\Phi \left( x\right)}=1+o(1) \ \ \ \ \ \
\textrm{and}  \ \ \  \ \ \
\frac{\mathbf{P}(S_n/  B_n \leq -x)}{\Phi \left( -x\right) } =1+o(1).
\end{equation}
Expansion is available for all $0\leq x =o(n^{1/2})$ if the moment generating function exists.
We refer to Chapter VIII of \citep{Petrov75} for further details  on the subject.

However,  the limit theorems for  self-normalized   partial sums
of independent random variables have put a new countenance on the classical limit theorems. The study of self-normalized  partial sums $S_n/V_n$
originates from Student's $t$-statistic.  Student's $t$-statistic $T_n$ is defined by
\[
T_n = \sqrt{n} \, \overline{X}_n / \widehat{\sigma},
\]
where $$\overline{X}_n = \frac{S_n}{n}  \ \ \ \textrm{and}\  \ \ \widehat{\sigma}^2 = \sum_{i=1}^n  \frac{(X_i - \overline{X}_n )^2}{ n-1}  .$$ It is known that for all $x\geq0,$
\[
\mathbf{P}\Big( T_n  \geq x \Big) = \mathbf{P}\bigg(  S_n/V_n  \geq x \Big(\frac{n}{n+x^2-1} \Big)^{1/2}  \bigg ),
\]
see  \citep{E69}. So, if we get an asymptotic bound on the tail probabilities for self-normalized partial sums,
then we  have an asymptotic bound on the tail probabilities for $T_n.$
 \citep{GGM97} gave a necessary and sufficient condition for the asymptotic
normality.  \citep{BBG96} (see also  \citep{BG96}) obtained the exact Berry-Esseen bound for self-normalized   partial sums.  \citep{S97} established a self-normalized Cram\'{e}r-Chernoff large deviation without any moment assumptions and
 \citep{S99} proved a self-normalized Cram\'{e}r moderate deviation theorem under  $(2+\rho)$th moments.
If $(X_i)_{i\geq1}$ are independent and identically distributed with  $\mathbf{E}|X_1|^{2+\rho}< \infty, \rho \in (0, 1] ,$ then
for all $0\leq x =o(n^{\rho/(4+2\rho)})$  as $n\rightarrow \infty,$
\begin{equation}\label{Cramer01}
\frac{\mathbf{P}(S_n/ V_n \geq x)}{1-\Phi \left( x\right)}=1+o(1).
\end{equation}
For  symmetric  independent random variables with finite third moments,
 \citep{WJ99}  derived an   exponential
nonuniform Berry-Esseen bound, while   \citep{CG03} further
 refined Wang and Jing's result    and obtained the following
Cram\'{e}r type moderate deviation expansion:
\begin{eqnarray} \label{fsfs123d}
\frac{\mathbf{P}(S_n/ V_n \geq x)}{1- \Phi(x)}  = 1 + O(1) (1+x)^3 B_n^{-3}\sum_{i=1}^n\mathbf{E}|X_i|^3,
\end{eqnarray}
where $O(1)$ is bounded by an absolute constant.
The expansion (\ref{fsfs123d}) was further extended to independent but not necessarily identically distributed random variables by
 \citep{JSW03} under  finite $(2+\rho)$th moments, $\rho \in (0,  1]$,
showing that
\begin{equation}\label{Cramer02}
\frac{\mathbf{P}(S_n/ V_n \geq x)}{1-\Phi \left( x\right)}=\exp\Big\{   O\big( 1\big)(1+x)^{2+\rho}   \varepsilon_{n }^\rho \Big\}
\end{equation}
uniformly for   $0\leq x =o( \min\{ \varepsilon_{n }^{-1}, \ \kappa_n^{-1} \}), $ where
\begin{eqnarray}\label{fgls1.5}
 \varepsilon_{n }^\rho =\sum_{i=1}^n \mathbf{E}|X_i|^{2+\rho} /B_n^{2+\rho}  \ \ \ \textrm{and} \ \ \  \kappa_n^2=\max_{1\leq i\leq n } \mathbf{E}X_i^2 / B_n^2.
\end{eqnarray}
For further self-normalized Cram\'{e}r type moderate deviation results  for independent random variables
we refer, for example, to    \citep{HSW09},    \citep{LSW13}, and  \citep{SZ16}.
  We also refer to  \citep{DLS09}  and  \citep{SW13}   for recent developments in this area. 

The theory for self-normalized sums of independent random variables
has been studied in depth. However, we are not aware of any such results for martingales.
  For some closely related topic, that is,  exponential inequalities for self-normalized martingales,
we refer to \citep{D99}, \citep{BT08}, \citep{CWM14} and \citep{BDR15}.
The main purpose of this paper is to establish self-normalized Cram\'{e}r type moderate deviation results for martingales.
Let  $(\delta_n)_{n\geq1}  , (\varepsilon_n)_{n\geq1}$ and $(\kappa_n)_{n\geq1}$ be three sequences of nonnegative numbers, such that
$\delta_n \rightarrow 0,$ $\varepsilon_n\rightarrow 0$ and $\kappa_n\rightarrow 0$ as $n\rightarrow \infty.$
Let $(X_i,\mathcal{F}_i)_{i\geq1}$ be a  sequence of martingale differences satisfying
$$\Big|  \sum_{i=1}^n \mathbf{E} [ X _i^2  |  \mathcal{F}_{i-1} ] -B_n^2  \Big| \leq \delta_n^2 B_n^2,$$
$$    \sum_{i=1}^n\mathbf{E}[|X_{i}| ^{2+\rho}  | \mathcal{F}_{i-1}]   \leq  \varepsilon_{n }^\rho    B_n ^{2+ \rho},   $$
and
$$ \max_{1\leq i \leq n} \mathbf{E}[ X_{i} ^2 | \mathcal{F}_{i-1}] \leq \kappa_n^2 B_n^2 \, ,    $$
 where $\rho \in (0, \frac32].$
From Corollary \ref{corollary01} we have
\begin{equation}\label{ghlm}
\mathbf{P}(S_n/V_n \geq x)=(1-\Phi(x))(1+o(1))
\end{equation}
uniformly for   $0\leq x =  o (\, \min\{  \varepsilon_n^{-\rho/(3+\rho)} , \ \delta_n^{-1},\ \kappa_n^{- 1} \})  $ as $n \rightarrow \infty.$
A more general Cram\'er type expansion is obtained in a larger range in our Theorem \ref{th3.1},
from which we  derive a moderate deviation principle for self-normalized martingales.
Moreover, when the condition $\sum_{i=1}^n\mathbf{E}[|X_{i}| ^{2+\rho}  | \mathcal{F}_{i-1}]   \leq  \varepsilon_{n }^\rho B_n^{2+\rho}$ is replaced by a slightly stronger condition
$$\mathbf{E}[|X_{i}| ^{2+\rho}  | \mathcal{F}_{i-1}]   \leq  (\varepsilon_{n }  B_n)^{\rho} \mathbf{E}[ X_{i}  ^{2}  | \mathcal{F}_{i-1}], $$
equality (\ref{ghlm}) holds for a  larger range of
$0\leq x =  o (\, \min\{  \varepsilon_n^{-\rho/(4+2\rho)} , \ \delta_n^{-1} \})$ for $\rho \in (0, 1]$, see Corollary \ref{comore}.
Clearly, our results recover (\ref{Cramer01}) for i.i.d.\ random variables.

 The rest of the paper is organized as follows.
Our main results are stated and discussed in Section \ref{sec2}.
Section \ref{sec3} provides the preliminary lemmas that are used in the proofs of the main results.
In Section  \ref{sec5},  we prove  the main results.

Throughout the paper the symbols $c$ and $c_\alpha,$  probably  supplied  with some indices,
denote respectively a generic positive absolute constant and a generic positive constant depending only on $\alpha.$

\section{Main results}\label{sec2}
\setcounter{equation}{0}
Let $(X_i,\mathcal{F}_i)_{i=0,...,n} $  be a sequence of   martingale differences defined on a
 probability space $(\Omega ,\mathcal{F},\mathbf{P})$,  where $X_0=0 $ and
 $\{\emptyset, \Omega\}=\mathcal{F}_0\subseteq ...\subseteq \mathcal{F}_n\subseteq
\mathcal{F}$ are increasing $\sigma$-fields. Set
\begin{equation}
S_{0}=0,\ \ \ \ \ S_k=\sum_{i=1}^k X_i,\quad k=1,...,n.  \label{xk}
\end{equation}
Then $S=(S_k,\mathcal{F}_k)_{k=0,...,n}$ is a martingale. Denote $B_n^2=\sum_{i=1}^n\mathbf{E}X_i^2$.
Let $[ S]$ and $\langle S \rangle$  be, respectively, the squared variance and the conditional variance  of the
martingale $S,$ that is
\begin{eqnarray*}
[S]_0=0,\ \ \ \ \ [ S]_k=\sum_{i=1}^k X_i^2,\quad k=1,...,n,
\end{eqnarray*}
and
\begin{equation}\label{quad01}
\langle S\rangle_0=0,\ \ \ \ \ \langle S \rangle_k=\sum_{i=1}^k \mathbf{E} [ X_i^2  |  \mathcal{F}_{i-1} ]  ,\quad k=1,...,n.
\end{equation}

In the sequel, we  use the following conditions:
\begin{description}
\item[(A1)]  There exists   $ \delta_n \in [0, \frac14] $ such that
$$ \Big| \sum_{i=1}^n\mathbf{E}[ X_{i}^2  | \mathcal{F}_{i-1}] - B_n^2 \Big| \leq  \delta_n^2 B_n^2 ;$$
\item[(A2)]  There exist   $\rho >0$ and  $ \varepsilon_n \in (0, \frac14]  $ such that
\[
\sum_{i=1}^n\mathbf{E}[|X_{i}| ^{2+\rho}  | \mathcal{F}_{i-1}]\leq   \varepsilon_n^\rho  B_n^{2+\rho};
\]
\item[(A3)]  There exists $ \kappa_n \in (0, \frac14]  $ such that
\[
\mathbf{E}[ X_{i}  ^2  | \mathcal{F}_{i-1}]\leq  \kappa_n^2 B_n^2, \ \ \ 1\leq i \leq n;
\]
\item[(A4)]  There exist  $\rho \in (0, 1]$ and   $ \gamma_n \in (0, \frac14]  $ such that
\[
\mathbf{E}[|X_{i}| ^{2+\rho}  | \mathcal{F}_{i-1}]   \leq  (\gamma_{n }  B_n)^{\rho}\, \mathbf{E}[ X_{i}  ^{2}  | \mathcal{F}_{i-1}], \ \ \ 1\leq i \leq n.
\]
\end{description}

When $\rho \in (0 , 1]$   and $\gamma_n \leq  (16/17)^{1/\rho}/4$,  conditions (A1) and (A4) imply condition (A2) with $\varepsilon_n=(17/16)^{1/\rho}\gamma_n.$
 Thus, we may assume that $\varepsilon_n =O(1)  \gamma_n$ as $n\rightarrow \infty.$
It is also easy to see that condition (A4) implies
condition (A3) with $\kappa_n=\gamma_n, $ see  Lemma \ref{lem3.1}.

In practice, we usually have $\max\{\delta_n, \varepsilon_n,  \gamma_n, \kappa_n \} \rightarrow 0 $ as $n\rightarrow \infty$.
In the case of  sums of i.i.d.\ random  variables,  conditions (A1), (A2), (A3), and (A4) are satisfied with  $\delta_n=0,$ $ \varepsilon_n, \gamma_n, \kappa_n= O(\frac {1} { \sqrt n}).$

Our first main result is the following Cram\'{e}r type moderate deviation for the self-normalized  martingale
$$W_n= S_n /  \sqrt{[ S]_n},$$
under  conditions (A1),  (A2), and (A3).  

\begin{thm}\label{th3.1}
Assume that   conditions (A1),  (A2), and (A3) are satisfied. Set $$\rho_1=\min\{ \rho,\, 1\}.$$ Then for all $0\leq x =o(\max\{ \varepsilon_n^{-1},  \kappa_n^{-1}  \})$,
\begin{eqnarray}\label{rest2.3}
\frac{\mathbf{P}(W_n \geq x)}{1-\Phi \left( x\right)}  =  \exp\bigg\{ \theta c_{\rho} \Big( x^{2+ \rho_1}  \varepsilon_n^{ \rho_1}+ x^2 \delta_n^2 +(1+x)\big( \varepsilon_n^{\rho/(3+\rho)}  +\delta_n \big) \Big) \bigg \}.
\end{eqnarray}
 Moreover,  the   equality remains valid when $\frac{\mathbf{P}(W_n \geq x)}{1-\Phi \left( x\right)}$ is replaced by $\frac{\mathbf{P}(W_n \leq -x)}{ \Phi \left( -x\right)}$.
\end{thm}

Under condition (A2) the best Berry-Esseen bound for standardized martingales is provided by   \citep{H88}.
Assuming $\langle S \rangle_n=B_n^2 $ a.s.,  Haeusler   proved that
\begin{eqnarray*}
\sup_x \Big| \mathbf{P}(S_n /B_n\leq x) -   \Phi \left( x\right)  \Big| \leq C  \Big( \sum_{i=1}^n\mathbf{E}|X_i/B_n
|^{2+\rho}  \Big)^{1/(3+\rho)}  .
\end{eqnarray*}
Moreover, it was showed that  this bound cannot be improved  for martingales with finite $(2+\rho)$th moments. In fact, there exist  positive absolute constant $c$  and a sequence of martingale differences   satisfying
$\mathbf{P}( S_{n} \leq 0) - \Phi \left( 0 \right) \geq c \, \big( \sum_{i=1}^n\mathbf{E}| X_{i}/B_n
|^{2+\rho}  \big)^{1/(3+\rho)} $
for all large enough $n$.
In particular, under  conditions  (A2) and $\langle S \rangle_n=B_n^2 $ a.s., Haeusler's result implies that
\begin{eqnarray}\label{hauslerb}
\sup_x \Big| \mathbf{P}(S_n /B_n\leq x) -   \Phi \left( x\right)  \Big| \leq C  \varepsilon_n^{\rho/(3+\rho)}.
\end{eqnarray}
Notice that
Theorem \ref{th3.1} implies that
\begin{eqnarray}\label{self-martingale}
\sup_x \big| \mathbf{P}( W_n \leq x) -   \Phi \left( x\right)  \big| \leq C \big( \varepsilon_n^{\rho/(3+\rho)} + \delta_n \big).
\end{eqnarray}
Under  conditions (A2) and $\langle S \rangle_n=B_n^2 $ a.s.,
the Berry-Esseen bound in (\ref{self-martingale}) for self-normalized martingales is of the same order as the  Berry-Esseen bound  in (\ref{hauslerb}) for standardized martingales.

From Theorem \ref{th3.1}, we obtain the following result about the equivalence to the normal tail.
\begin{corollary}\label{corollary01}
Assume that conditions (A1),  (A2), and (A3) are satisfied with $\rho \in (0, \frac32]$. Then
\begin{eqnarray*}
\frac{\mathbf{P}(W_n \geq x)}{1-\Phi \left( x\right)}=1+o(1) \ \ \ \ \ \
\textrm{and}  \ \ \  \ \ \
\frac{\mathbf{P}(W_n \leq -x)}{\Phi \left( -x\right) } =1+o(1)
\end{eqnarray*}
uniformly for  $0\leq x = o (\, \min\{  \varepsilon_n^{-\rho/(3+\rho)} , \,   \kappa_n^{-1} , \, \delta_n^{-1}   \})$ as $ n \rightarrow \infty.$
\end{corollary}

Theorem \ref{th3.1}  also implies  the following moderate  deviation principles (MDP) for self-normalized martingales.
\begin{corollary}\label{corollary02}
Assume  conditions (A1), (A2), and (A3)  with $\max\{\delta_n, \varepsilon_n,  \kappa_n\} \rightarrow 0$ as $n\rightarrow \infty$.
Let $a_n$ be any sequence of real numbers satisfying $a_n \rightarrow \infty$ and $a_n\varepsilon_n\rightarrow 0$
as $n\rightarrow \infty$.  Then  for each Borel set $B$,
\begin{eqnarray}
- \inf_{x \in B^o}\frac{x^2}{2} &\leq & \liminf_{n\rightarrow \infty}\frac{1}{a_n^2}\ln \mathbf{P}\bigg(\frac{  W_n}{a_n }  \in B \bigg) \nonumber \\
 &\leq& \limsup_{n\rightarrow \infty}\frac{1}{a_n^2}\ln \mathbf{P}\bigg(\frac{ W_n}{a_n }    \in B \bigg) \leq  - \inf_{x \in \overline{B}}\frac{x^2}{2} \, ,   \label{MDP}
\end{eqnarray}
where $B^o$ and $\overline{B}$ denote the interior and the closure of $B$, respectively.
\end{corollary}

The last corollary shows that the convergence speed of MDP  depends only on $\varepsilon_n$ and it has nothing to do with the convergence speeds  of $\kappa_n$ and $ \delta_n$.

For i.i.d.\ random variables, the self-normalized MDP  was established  by  \citep{S97}.
See also   \citep{JLZ12} for  non-identically distributed  random variables.

The other main results concern some improvements of   Theorem \ref{th3.1}
when condition (A3)   is replaced by  the stronger condition (A4). Theorems
\ref{th1} and \ref{th2} below give respectively lower and upper bounds, while Theorem \ref{rezultCRMD} gives a
Cram\'{e}r type moderate deviation expansion sharper than that in  Theorem \ref{th3.1}.


\begin{thm}\label{th1}
Assume that conditions  (A1),  (A2), and (A4) are satisfied.
\begin{description}
  \item[\textbf{[i]}]  If $\rho \in (0, 1)$, then for all $0\leq x =o(\gamma_n^{-1})$,
\begin{equation} \label{t1ie1}
\frac{\mathbf{P}(W_n \geq x)}{1-\Phi \left( x\right)}\geq \exp\bigg\{ -c_{\rho  } \bigg( x^{2+\rho}  \varepsilon_n^\rho+ x^2 \delta_n^2 +(1+x)\left(  x^\rho \gamma_n^\rho +   \gamma_n^\rho      +  \delta_n   \right) \bigg) \bigg \}. \ \ \
\end{equation}
  \item[\textbf{[ii]}]   If $\rho =1$, then  for all $0\leq x  =o(\gamma_n^{-1})$,
\begin{equation}\label{t1ie12}
\frac{\mathbf{P}(W_n \geq x)}{1-\Phi \left( x\right)}\geq \exp\bigg\{ -c  \bigg( x^{3}  \varepsilon_n + x^2 \delta_n^2+ (1+x)\left(x \gamma_n+  \gamma_n|\ln \gamma_n|     +  \delta_n   \right) \bigg) \bigg \}.
\end{equation}
\end{description}
 Moreover,  the two equalities above remain valid when $\frac{\mathbf{P}(W_n \geq x)}{1-\Phi \left( x\right)}$ is replaced by $\frac{\mathbf{P}(W_n \leq -x)}{ \Phi \left( -x\right)}$.
\end{thm}

For any sequence of positive numbers $(\alpha_n)_{n\geq 1}$  denote
 \begin{equation}
 \widehat{\alpha}_n(x, \rho) =  \frac{ \alpha_n^{ \rho(2-\rho)/4 } }{ 1+ x  ^{  \rho(2+\rho)/4 }}.
 \end{equation}
\begin{thm}\label{th2}
Assume that conditions (A1),  (A2), and (A4) are satisfied.
\begin{description}
  \item[\textbf{[i]}] If $\rho \in (0, 1)$, then for all $0 \leq x =o(\gamma_n^{-1})$,
   \begin{equation*}
\frac{\mathbf{P}(W_n \geq x)}{1-\Phi \left( x\right)}\leq  \exp\bigg\{  c_{\rho  } \bigg( x^{2+\rho}  \varepsilon_n^\rho+ x^2 \delta_n^2 +(1+x)\Big(  x^\rho \gamma_n^\rho +   \gamma_n^\rho      +  \delta_n +  \widehat{\varepsilon}_n(x, \rho) \Big) \bigg) \bigg \}.
\end{equation*}

\item[\textbf{[ii]}] If $\rho =1,  $  then for all $0 \leq x =o(\gamma_n^{-1})$,
\begin{eqnarray*}
\frac{\mathbf{P}(W_n \geq x)}{1-\Phi \left( x\right)}  \leq \exp\bigg\{  c \,  \bigg( x^{3}  \varepsilon_n + x^2 \delta_n^2+ (1+x)\Big(x \gamma_n+  \gamma_n|\ln \gamma_n|     +  \delta_n  + \widehat{\varepsilon}_n(x, 1) \Big) \bigg) \bigg \} .
\end{eqnarray*}
\end{description}
 Moreover, the two equalities above remain valid when $\frac{\mathbf{P}(W_n \geq x)}{1-\Phi \left( x\right)}$ is replaced by $\frac{\mathbf{P}(W_n \leq -x)}{ \Phi \left( -x\right)}$.
\end{thm}

Combining Theorems \ref{th1} and \ref{th2}, we obtain  the following    Cram\'{e}r type moderate deviation expansion for self-normalized martingales under conditions (A1),  (A2), and (A4),  which
 is stronger than the expansion in  Theorem \ref{th3.1}  since the term $\varepsilon_n^{\rho/(3+\rho)}$ therein is improved
 to a smaller one. In what follows, $\theta $  stands for values satisfying $\left| \theta  \right| \leq 1$.

\begin{thm}\label{rezultCRMD}
Assume that conditions (A1),  (A2), and (A4) are satisfied.
\begin{description}
\item[\textbf{[i]}] If $\rho \in (0, 1)$, then for all $0\leq x =o(\gamma_n^{-1})$,
\begin{eqnarray*}
\frac{\mathbf{P}(W_n \geq x)}{1-\Phi \left( x\right)}  =  \exp\bigg\{ \theta c_{\rho} \bigg( x^{2+\rho}  \varepsilon_n^\rho+ x^2 \delta_n^2 +(1+x)\Big(x^\rho \gamma_n^\rho + \gamma_n ^\rho   + \delta_n +  \widehat{\varepsilon}_n(x, \rho)\Big) \bigg) \bigg \} .
\end{eqnarray*}
\item[\textbf{[ii]}] If $\rho =1$, then for all $0\leq x =o(\gamma_n^{-1})$,
\begin{eqnarray*}
 \frac{\mathbf{P}(W_n \geq x)}{1-\Phi \left( x\right)}  = \exp\bigg\{ \theta c  \bigg( x^{3}  \varepsilon_n + x^2 \delta_n^2+(1+x)\Big( x  \gamma_n  + \gamma_n |\ln \gamma_n |    + \delta_n+ \widehat{\varepsilon}_n(x, 1)  \Big) \bigg) \bigg \} .
\end{eqnarray*}
\end{description}
 Moreover,  the two equalities above remain valid when $\frac{\mathbf{P}(W_n \geq x)}{1-\Phi \left( x\right)}$ is replaced  by $\frac{\mathbf{P}(W_n \leq -x)}{ \Phi \left( -x\right)}$.
\end{thm}

Notice that condition (A4) implies condition (A2) with $\varepsilon_n=\gamma_n.$  Therefore, it follows from Theorem
 \ref{rezultCRMD} that:
\begin{corollary}\label{rezultCRMD2}
Assume that conditions (A1) and (A4) are satisfied.
\begin{description}
  \item[\textbf{[i]}] If $\rho \in (0, 1)$, then for all $0\leq x =o(\gamma_n^{-1})$,
\begin{eqnarray*}
\frac{\mathbf{P}(W_n \geq x)}{1-\Phi \left( x\right)}  =  \exp\bigg\{ \theta c_{\rho} \bigg( x^{2+\rho}  \gamma_n^\rho+ x^2 \delta_n^2 +(1+x)\Big(  \delta_n + \widehat{\gamma}_n(x, \rho) \Big) \bigg) \bigg \} .
\end{eqnarray*}
 \item[\textbf{[ii]}] If $\rho =1$, then for all $0\leq x =o(\gamma_n^{-1})$,
\begin{eqnarray*}
\frac{\mathbf{P}(W_n \geq x)}{1-\Phi \left( x\right)}  = \exp\bigg\{ \theta c  \bigg( x^{3}  \gamma_n + x^2 \delta_n^2+(1+x)\Big(   \delta_n+ \gamma_n |\ln \gamma_n|  +\widehat{\gamma}_n(x, 1) \Big) \bigg) \bigg \} .
\end{eqnarray*}
\end{description}
 Moreover,   the two equalities above remain valid when $\frac{\mathbf{P}(W_n \geq x)}{1-\Phi \left( x\right)}$ is replaced  by $\frac{\mathbf{P}(W_n \leq -x)}{ \Phi \left( -x\right)}$.
\end{corollary}

From Theorem \ref{rezultCRMD}, we also obtain the following result about the equivalence to the normal tail.
\begin{corollary}\label{comore}
Assume  conditions (A1),  (A2), and (A4) with $\rho \in (0 , 1]$. Then
\begin{equation}\label{Cramer}
\frac{\mathbf{P}(W_n \geq x)}{1-\Phi \left( x\right)}=1+o(1) \ \ \ \ \ \
\textrm{and}  \ \ \  \ \ \
\frac{\mathbf{P}(W_n \leq -x)}{\Phi \left( -x\right) } =1+o(1)
\end{equation}
 uniformly for   $0\leq x = o (\, \min\{  \varepsilon_n^{-\rho/(2+\rho)} , \ \gamma_n^{- \rho/(1+\rho)}, \ \delta_n^{-1}  \})$ as $n \rightarrow \infty.$
\end{corollary}

In the case of   i.i.d.\ random variables,  conditions (A1), (A2), and (A4) are satisfied with $\varepsilon_n , \gamma_n= O(1/\! \sqrt{n}\,)$  and $\delta_n=0$.
  Thus, the range $0\leq x = o (\, \min\{  \varepsilon_n^{-\rho/(2+\rho)} , \ \delta_n^{-1},  \ \gamma_n^{-\rho/(1+\rho)}  \})$ reduces to $0\leq x = o (n^{-\rho/(4+2\rho )}   ), n \rightarrow \infty,$ which is the best possible result such that (\ref{Cramer}) holds (see   \citep{S99}).  Moreover, from Theorem \ref{rezultCRMD}, we can get the  estimation of the rate of convergence in (\ref{Cramer});  for example, when $\rho = 1$ we have:

\begin{corollary}\label{comore-b}
Assume  conditions (A1),  (A2), and (A4) with $\rho = 1$, $\varepsilon_n,  \gamma_n, \delta_n= O(1/\! \sqrt{n}\,).$   Then,  for $x= x_0 n^{ \frac{1}{2} - a}$ with $0<a< \frac{4}{11}$ and $x_0 >0$ fixed, as $ n \rightarrow \infty,$
\begin{equation}\label{Cramer-comore-b}
\frac{\mathbf{P}(W_n \geq x)}{1-\Phi \left( x\right)}=\exp \Big\{ O(1)\frac { x^3} {\sqrt n}  \Big\}   \ \ \
\textrm{and}  \ \ \  \
\frac{\mathbf{P}(W_n \leq -x)}{\Phi \left( -x\right) } =\exp \Big \{  O(1) \frac { x^3} {\sqrt n} \Big\}.
\end{equation}
In particular,  for $x= x_0 n^{ \frac{1}{6} -  b }$ with $0<b<\frac{1}{33}$ and $x_0 >0$ fixed, as $ n \rightarrow \infty,$
\begin{equation}
\frac{\mathbf{P}(W_n \geq x)}{1-\Phi \left( x\right)}=1+ O\Big(\frac { x^3} {\sqrt n} \Big)    \ \ \
\textrm{and}  \ \ \  \
\frac{\mathbf{P}(W_n \leq -x)}{\Phi \left( -x\right) } =1+ O\Big( \frac { x^3} {\sqrt n} \Big) .
\end{equation}
\end{corollary}
Notice that the rate  of convergence in (\ref{Cramer-comore-b})  coincides with that in (\ref{Cramer02}) for   i.i.d.\ random variables.

\section{ \textbf{Preliminary lemmas} }\label{sec3}
\setcounter{equation}{0}
 The proofs of  Theorems  \ref{th3.1}-\ref{rezultCRMD} are  based on a conjugate multiplicative martingale technique for changing the probability
measure which is similar to that of the transformation of \citep{E24}.
Our  approach is inspired by the earlier work of  \citep{GH00} on Cram\'{e}r moderate deviations
for standardized martingales, and by that of   \citep{S99},  \citep{JSW03},
who developed techniques for moderate deviations of self-normalized sums of independent random variables.
We extend these work by introducing a new choice of the density for the change of measure
and  refining the approaches in \citep{S99} and \citep{JSW03} to handle self-normalized martingales.
A key point of the proof is a new Berry-Esseen bound
for martingales under the changed measure, see Proposition  \ref{lem3.4} below.

Let $$ \xi_i = \frac{X_i}{ B_n}, \ \ \ \ i=1,...,n.$$ Then  $(\xi _i,\mathcal{F}_i)_{i=0,...,n}$  is also a sequence  of martingale differences.
Moreover,   for simplicity of notations, set
\begin{equation*}
M_k=\sum_{i=1}^k \xi_i,
\end{equation*}
\begin{equation*}
[M]_k=\sum_{i=1}^k \xi_i^2 \ \ \ \  \textrm{and} \ \ \ \   \langle M \rangle_k=\sum_{i=1}^k \mathbf{E} [ \xi_i^2  |  \mathcal{F}_{i-1} ], \quad k=1,...,n.
\end{equation*}
Thus
\begin{equation}\label{denow}
W_n= \frac{S_n}{ \sqrt{[ S]_n}} = \frac{ M_n}{\sqrt{[M]_n}}.
\end{equation}

For any real number $\lambda$,
consider the \emph{exponential multiplicative martingale} $Z(\lambda
)=(Z_k(\lambda ),\mathcal{F}_k)_{k=0,...,n},$ where
\[
Z_0(\lambda )=1,\quad    Z_k(\lambda )=\prod_{i=1}^k\frac{e^{\zeta_i(\lambda)}}{\mathbf{E}[e^{\zeta_i(\lambda)}|
\mathcal{F}_{i-1}]},\quad k=1,...,n  \label{C-1}
\]
with
$$ \zeta_i(\lambda)=\lambda \xi_i - \lambda^2 \xi_i^2/2.$$
Thus, for each
real number $\lambda$ and each $k=1,...,n,$ the random variable $Z_k(\lambda
) $ is nonnegative and $\mathbf{E} Z_k(\lambda )=1.$  The last
observation allows us to introduce   the \emph{conjugate probability measure} $\mathbf{P}_\lambda := \mathbf{P}_{\lambda,n}$ on $(\Omega ,%
\mathcal{F})$ defined by
\begin{equation}
d\mathbf{P}_\lambda =Z_n(\lambda )d\mathbf{P}.  \label{f21}
\end{equation}
Although $(M_k,\mathcal{F}_k)_{k=0,...,n}$ is a martingale under the measure $\mathbf{P},$ it is no longer a
martingale under the conjugate probability measure $\mathbf{P}_\lambda.$
To obtain a martingale under $\mathbf{P}_\lambda$ we have to  center the random variables $\zeta_i(\lambda)$.
Denote by $\mathbf{E}_{\lambda}$ the expectation with respect to $\mathbf{P}_{\lambda}$. Because $Z(\lambda )$ is a
 uniformly integrable martingale under $\mathbf{P}$, we have
\begin{eqnarray}
 \mathbf{E}_\lambda [\zeta] =\mathbf{E}[\zeta Z_i(\lambda )]
\end{eqnarray}
and
\begin{eqnarray}\label{sfghkm423}
\mathbf{E}_{\lambda}[\zeta|
\mathcal{F}_{i-1}]=   \frac{ \mathbf{E}[\zeta e^{\zeta_i(\lambda)}|\mathcal{F}_{i-1}]}{\mathbf{E}[e^{\zeta_i(\lambda)} |\mathcal{F}_{i-1}] }
\end{eqnarray}
for any $\mathcal{F}_i$-measurable
random variable $\zeta$  that is integrable with respect to $\mathcal{F}_i.$ Set
\[
b_i(\lambda)=\mathbf{E}_{\lambda}[\zeta_i(\lambda) |\mathcal{F}_{i-1}],\quad i=1,\dots,n,
\]
\[
\eta_i(\lambda)= \zeta_i(\lambda) - b_i(\lambda), \quad i=1,\dots,n,
\]
and
\begin{equation}\label{f23}
Y_k(\lambda )=\sum_{i=1}^k\eta _i(\lambda ),\quad k=1,...,n.
\end{equation}
Then $Y(\lambda )=(Y_k(\lambda ),\mathcal{F}_k)_{k=0,...,n}$ is the \emph{conjugate martingale}.
The following semimartingale decomposition is well-known:
\begin{equation}
 \sum_{i=1}^k\zeta_i(\lambda)  =B_k(\lambda )+Y_k(\lambda ),\quad k=1,...,n, \label{xb}
\end{equation}
where $B(\lambda )=(B_k(\lambda ),\mathcal{F}_k)_{k=0,...,n}$ is the \emph{%
drift process} defined as
\[
B_k(\lambda )=\sum_{i=1}^kb_i(\lambda ),\quad k=1,...,n.
\]
By the relation between $\mathbf{E}$ and $\mathbf{E}_{\lambda}$ on $\mathcal{F}_i,$
we have
\begin{equation}\label{brads}
b_i(\lambda )=\frac{\mathbf{E}[\zeta_i(\lambda)e^{\zeta_i(\lambda)}|\mathcal{F}_{i-1}]}{\mathbf{E}[e^{\zeta_i(\lambda)}|\mathcal{F}%
_{i-1}]},\quad i=1,...,n.
\end{equation}
It is easy to compute the conditional variance of  the conjugate
martingale $Y(\lambda )$ under the measure $\mathbf{P}_\lambda $, for $k=0,...,n$,
\begin{eqnarray}
\left\langle Y(\lambda )\right\rangle _k &=& \sum_{i=1}^{k}\mathbf{E}_{\lambda}[\eta_{i}(\lambda)^2 | \mathcal{F}_{i-1}] \nonumber\\
& = &  \sum_{i=1}^{k}\mathbf{E}_{\lambda}[(\zeta_i(\lambda) - b_i(\lambda))^2  | \mathcal{F}_{i-1}] \nonumber\\
&= &\sum_{i=1}^k\left( \frac{\mathbf{E}[\zeta_i^2(\lambda) e^{\zeta_i(\lambda)}|\mathcal{F}_{i-1}]}{\mathbf{E}[e^{\zeta_i(\lambda)}|\mathcal{F}%
_{i-1}]}-\frac{\mathbf{E}[\zeta_i(\lambda)e^{\zeta_i(\lambda) }|\mathcal{F}_{i-1}]^2}{\mathbf{E}[e^{\zeta_i(\lambda)}|\mathcal{F}_{i-1}]^2}\right). \label{f24}
\end{eqnarray}

In the sequel, we  give the upper and lower bounds for $B_n(\lambda ).$ To this end, we need the following three useful
lemmas. The proof is similar to tat in  \citep{S99}  and \citep{JSW03}. Set
 \[
\widetilde{\varepsilon}_\lambda = \lambda^2\mathbf{E}[ \xi_i^2\mathbf{1}_{\{ |\lambda \xi_i| > 1\}} |  \mathcal{F}_{i-1}] + \lambda^3 \mathbf{E}[\xi_i^3\mathbf{1}_{\{ |\lambda \xi_i| \leq 1\}} |\mathcal{F}_{i-1} ],\ \ \ \lambda\geq 0.
\]
If $\mathbf{E}[ |\xi_i|^{2+\rho} ] < \infty$ for  $ \rho \in [0, 1], $
then it is obvious that  $$\widetilde{\varepsilon}_\lambda \leq  \lambda^{2+\rho}  \mathbf{E}[ |\xi_i|^{2+\rho}  |  \mathcal{F}_{i-1}],\ \ \lambda \geq 0.$$

\begin{lemma}\label{lem001}
For all $\lambda > 0$ and  $ \tau\in [\frac18, 2] ,$  we have
\[
  \mathbf{E}[e^{ \lambda \xi_i -\tau \lambda^2 \xi_i^2 }|  \mathcal{F}_{i-1}] = 1+ (\frac12  -\tau)\lambda^2\mathbf{E}[\xi_i^2|  \mathcal{F}_{i-1}]+O(1)  \widetilde{\varepsilon}_\lambda ,
\]
  where $O(1)$ is bounded by an absolute constant.
\end{lemma}
\begin{lemma}
\label{lem002}
   For all $\lambda > 0,$   we have
\begin{eqnarray*}
\mathbf{E}[e^{\zeta_i(\lambda)  } |  \mathcal{F}_{i-1}]&=& 1 +   O(1)  \widetilde{\varepsilon}_\lambda,\\
\mathbf{E}[\zeta_i(\lambda)   e^{\zeta_i(\lambda)  } |  \mathcal{F}_{i-1}]&=&  \frac12 \lambda^2 \mathbf{E}[ \xi_i^2|  \mathcal{F}_{i-1}] +  O(1)  \widetilde{\varepsilon}_\lambda ,\\
\mathbf{E}[\zeta_i ^2 (\lambda) e^{\zeta_i(\lambda)  } |  \mathcal{F}_{i-1}]&=&    \lambda^2 \mathbf{E}[ \xi_i^2|  \mathcal{F}_{i-1}] +  O(1)  \widetilde{\varepsilon}_\lambda , \\
\mathbf{E}[|\zeta_i(\lambda)|^3  e^{\zeta_i(\lambda)  }  |  \mathcal{F}_{i-1}]&=&   O(1)  \widetilde{\varepsilon}_\lambda,   \\
(\mathbf{E}[\zeta_i(\lambda)   e^{\zeta_i(\lambda)  } |  \mathcal{F}_{i-1}])^2&=&  O(1)  \widetilde{\varepsilon}_\lambda,
\end{eqnarray*}
where $O(1)$ is bounded by an absolute constant.
\end{lemma}

\begin{lemma}
\label{lem003}
Let $Z_i=\xi_i^2 -\mathbf{E}[\xi_i^2|  \mathcal{F}_{i-1}] .$  Then for all $\lambda >0,$
\begin{eqnarray*}
\mathbf{E}[Z_ie^{\zeta_i(\lambda)  }|  \mathcal{F}_{i-1} ]&=&   O(1) \frac{1}{\lambda^2} \widetilde{\varepsilon}_\lambda,\\
\mathbf{E}[Z_i^2   e^{\zeta_i (\lambda) } |  \mathcal{F}_{i-1}]&=&   O(1) \frac{1}{\lambda^4} \widetilde{\varepsilon}_\lambda,
\end{eqnarray*}
where $O(1)$ is bounded by an absolute constant.
\end{lemma}

Using  Lemma
\ref{lem002}, we obtain the following upper and lower bounds for $B_n(\lambda ).$
\begin{lemma}
\label{lem3.2} Assume   conditions  (A2) and (A3) with $\rho \in (0 , 1]$. Then
for all $0 \leq \lambda  =o(\max\{ \varepsilon_n^{-1},  \kappa_n^{-1}  \}),$
\begin{equation}\label{f25}
B_n(\lambda) =  \frac12 \lambda^2  \langle M \rangle_n +  O(1) \lambda^{2+\rho}  \varepsilon_n^\rho  ,
\end{equation}
where $O(1)$ is bounded by an absolute constant.
\end{lemma}
\begin{proof}  According to the definition of $b_i(\lambda ),$ we have
$$b_i(\lambda )=\frac{\mathbf{E}[\zeta_i(\lambda)e^{\zeta_i(\lambda)}|\mathcal{F}_{i-1}]}{\mathbf{E}[e^{\zeta_i(\lambda)}|\mathcal{F}%
_{i-1}]}  .$$
By Lemma \ref{lem002}, it follows that
$$\mathbf{E}[\zeta_i(\lambda)e^{\zeta_i(\lambda)}|\mathcal{F}_{i-1}]= \frac12 \lambda^2 \mathbf{E}[  \xi_i ^{2} |\mathcal{F}_{i-1}] +  O(1)\widetilde{\varepsilon}_\lambda$$
and
\begin{eqnarray}\label{fsd2}
\mathbf{E}[e^{\zeta_i(\lambda)}|\mathcal{F}_{i-1}] = 1+  O(1) \widetilde{\varepsilon}_\lambda .
\end{eqnarray}
Therefore,  conditions  (A2)  and (A3) imply that  for all $0 \leq \lambda  =o(\max\{ \varepsilon_n^{-1},  \kappa_n^{-1}  \} ),$
\begin{eqnarray*}
b_i(\lambda )= \frac12 \lambda^2 \mathbf{E}[  \xi_i ^{2} |\mathcal{F}_{i-1}] +  O(1) \widetilde{\varepsilon}_\lambda
\end{eqnarray*}
and
$$ B_n(\lambda) =  \frac12 \lambda^2  \langle M \rangle_n +  O(1) \lambda^{2+\rho}\varepsilon_n^\rho  $$
as desired.
\end{proof}

The following lemma shows that condition (A4) implies condition (A3) with $\kappa_n=\gamma_n.$
\begin{lemma}\label{lem3.1}
Assume condition (A4). Then $ \mathbf{E}[\xi_{i}^2 | \mathcal{F}_{i-1}]  \leq   \gamma_n^{2}.$
\end{lemma}
\begin{proof}
By Jensen's inequality and  condition (A4), it holds that
\begin{eqnarray*}
\mathbf{E}[\xi_{i}^2 | \mathcal{F}_{i-1}]^{(2+\rho)/2}  \leq   \mathbf{E}[|\xi_{i}|^{2+\rho} | \mathcal{F}_{i-1}]
 \leq   \gamma_n^\rho  \mathbf{E}[\xi_{i}^2|\mathcal{F}_{i-1}],
\end{eqnarray*}
from which we get $\mathbf{E}[\xi_{i}^2 | \mathcal{F}_{i-1}] \leq \gamma_n^2 .$
\end{proof}

\begin{lemma}\label{lema1} Assume condition (A4). Then  for any $ t \in [0,   \rho)$,
\begin{equation}
\mathbf{E} [|\xi_{i}|^{2+t} | \mathcal{F}_{i-1} ] \leq \gamma_n^{t } \, \mathbf{E} [\xi_{i}^{2} | \mathcal{F}_{i-1} ].
\end{equation}
\end{lemma}
\begin{proof}
Let $l, p, q$ be defined by the following equations
 $$lp=2, \ \ \ \ (2+t-l)q=2+\rho, \ \ \  p^{-1}+ q^{-1}=1, \ \ \ l>0,\ \textrm{and}\  p, q\geq1.$$
Solving the last equations, we get
 $$l= \frac{2(\rho-t)}{\rho },\ \ \ \ \  p=\frac{\rho }{\rho-t}, \ \ \ \ \  q=\frac{\rho }{t }.$$
By  H\"{o}lder's inequality and condition (A4), it is easy to see that
\begin{eqnarray*}
\mathbf{E} [|\xi_{i}|^{2+t} | \mathcal{F}_{i-1} ] &=& \mathbf{E} [|\xi_{i}|^{l} |\xi_{i}|^{2+t-l} | \mathcal{F}_{i-1} ] \\
&\leq &( \mathbf{E} [|\xi_{i}|^{lp}  | \mathcal{F}_{i-1} ])^{1/p} ( \mathbf{E} [  |\xi_{i}|^{(2+t-l)q} | \mathcal{F}_{i-1} ])^{1/q}\\
&\leq & ( \mathbf{E} [\xi_{i}^{2}  | \mathcal{F}_{i-1} ])^{1/p} ( \mathbf{E} [  |\xi_{i}|^{2+\rho} | \mathcal{F}_{i-1} ])^{1/q} \\
&\leq & ( \mathbf{E} [\xi_{i}^{2}  | \mathcal{F}_{i-1} ])^{1/p} (\gamma_n^{\rho } \mathbf{E} [ \xi_{i}^2 | \mathcal{F}_{i-1} ])^{1/q} \\
&\leq & \gamma_n^{ \rho /q} \mathbf{E} [\xi_{i}^{2}  | \mathcal{F}_{i-1} ]\\
&=& \gamma_n^{t } \ \mathbf{E} [\xi_{i}^{2} | \mathcal{F}_{i-1} ].
\end{eqnarray*}
This completes the proof of the lemma.
\end{proof}

\begin{lemma}\label{lema24} Assume conditions (A1) and  (A2). Then   for any $ t \in [0,   \rho)$,
\begin{equation}
\sum_{i=1}^n\mathbf{E}[|\xi_{i}| ^{2+t}  | \mathcal{F}_{i-1}]\leq  2\, \varepsilon_n^t  .
\end{equation}
\end{lemma}
\begin{proof}  Recall the notations in the proof of   Lemma \ref{lema1}. It is easy to see  that
\begin{eqnarray*}
\sum_{i=1}^n\mathbf{E}[|\xi_{i}| ^{2+t}  | \mathcal{F}_{i-1}]
 \leq   \sum_{i=1}^n ( \mathbf{E} [\xi_{i}^{2}  | \mathcal{F}_{i-1} ])^{1/p} ( \mathbf{E} [  |\xi_{i}|^{2+\rho} | \mathcal{F}_{i-1} ])^{1/q}.
\end{eqnarray*}
Using H\"{o}lder's inequality and conditions (A1) and (A2), we have
\begin{eqnarray*}
\sum_{i=1}^n\mathbf{E}[|\xi_{i}| ^{2+t}  | \mathcal{F}_{i-1}]
&\leq & \Big(\sum_{i=1}^n \mathbf{E} [\xi_{i}^{2}  | \mathcal{F}_{i-1} ] \Big)^{1/p} \Big ( \sum_{i=1}^n \mathbf{E} [  |\xi_{i}|^{2+\rho} | \mathcal{F}_{i-1} ] \Big)^{1/q}\\
&\leq & 2 \, \varepsilon_n^{t},
\end{eqnarray*}
which gives the desired inequality.
\end{proof}

We will also need the following two lemmas.
\begin{lemma}\label{slemms}   Assume  condition (A1). Then for all $x> 0,$
 \begin{eqnarray*}
  \mathbf{P}\Big(M_n \geq x\sqrt{[M]_n},\    [M]_n  \geq 16 \Big)
 \leq  \frac23 x^{-2/3} \exp\Big\{-\frac34 x^2  \Big\}.
\end{eqnarray*}
\end{lemma}
\begin{proof}  By inequality (11) of \citep{D09}, we have for all $\lambda \in \mathbf{R},$
\begin{eqnarray*}
  \mathbf{E}\exp\Big\{\lambda M_n - \frac{\lambda^2}{2}( \frac13 [M]_n + \frac23 \langle M\rangle_n )  \Big\}
 \leq  1.
\end{eqnarray*}
Applying the last inequality to the exponential inequality of \citep{DP09} with $p=q=2, $ we deduce that for all $x>0,$
 \begin{eqnarray}
  \mathbf{P}\Bigg(  \frac{|M_n|}{\sqrt{ \frac32  ( \frac13 [M]_n + \frac23 \langle M\rangle_n +\mathbf{E}M_n^2)} }  \geq  x  \Bigg)
 \leq  \Big(\frac23\Big)^{2/3} x^{-2/3} \exp\Big\{-\frac12 x^2  \Big\}.
\end{eqnarray}
By condition (A1) and the fact $\mathbf{E}\langle M \rangle_n  = \mathbf{E}M_n^2=1 $, it is easy to see that for all $x>0,$
 \begin{eqnarray}
  \mathbf{P}\Big(M_n \geq x\sqrt{[M]_n},\    [M]_n  \geq 16 \Big)
&\leq&  \mathbf{P}\Big(M_n \geq x\sqrt{\frac34[M]_n+ 4},\    [M]_n  \geq 16 \Big) \nonumber \\
&\leq&  \mathbf{P}\Big(M_n \geq x\sqrt{\frac34[M]_n   + \frac32\langle M \rangle_n + \frac94 \mathbf{E}M_n^2  },\    [M]_n  \geq 16 \Big) \nonumber \\
&\leq&  \mathbf{P}\Big(M_n \geq x\sqrt{\frac34[M]_n   + \frac32\langle M \rangle_n + \frac94\mathbf{E}M_n^2  } \, \Big) \nonumber \\
&=&  \mathbf{P}\bigg(M_n \geq  \sqrt{\frac32}  x \, \sqrt{\frac12[M]_n   +  \langle M \rangle_n + \frac32\mathbf{E}M_n^2  } \ \bigg) \nonumber \\
&\leq& \frac23 x^{-2/3} \exp\Big\{-\frac34 x^2  \Big\}  \nonumber
\end{eqnarray}
as desired.
\end{proof}

\begin{lemma}\label{slemma58}  Assume  conditions (A1) and   (A2). Then
$$\mathbf{P}\big(  | [M]_n - \langle M\rangle_n  | \geq 1 \big)\leq c_{\rho} \, \big(  \varepsilon_n ^{ (2+\rho)/2 }     + \varepsilon_n^\rho \,   \big)  .$$
\end{lemma}
\begin{proof}
Notice  that $   [M]_n - \langle M \rangle_n      =   \sum_{i=1}^n(\xi_i^2-\mathbf{E} [\xi_i^2 | \mathcal{F}_{i-1}]) $  is a martingale.
For $\rho$,   we distinguish two cases as follows.

When $\rho \in (0, 2],$  by the inequality  of \citep{V65},  it follows that
\begin{eqnarray}
\mathbf{E} [ |[M]_n - \langle M \rangle_n |^{ (2+\rho)/2 }]
  & \leq&    \sum_{i=1}^n\mathbf{E}[| \xi_i^2-\mathbf{E}[\xi_i^2 | \mathcal{F}_{i-1}]| ^{(2+\rho)/2}] \nonumber \\
& \leq&  c_1 \sum_{i=1}^n\mathbf{E}[|\xi_i| ^{ 2+\rho }  ]  \nonumber\\
& \leq&  c_2\,  \varepsilon_n^\rho, \nonumber
\end{eqnarray}
where the last line follows by conditions (A1) and   (A2).
Hence, by Markov's inequality,
\begin{eqnarray*}
\mathbf{P}\big(  | [M]_n - \langle M\rangle_n  | \geq 1 \big) &\leq& \mathbf{E} [ |[M]_n - \langle M \rangle_n |^{ (2+\rho)/2 }] \\
&\leq&  c_2 \, \varepsilon_n^\rho,
\end{eqnarray*}

When $\rho >2,$  by Rosenthal's inequality (cf.,  Theorem 2.12 of \citep{HH80}), Lemma \ref{lema24}, and condition (A2), it follows that
\begin{eqnarray}
&&\mathbf{E} [ |[M]_n - \langle M \rangle_n |^{ (2+\rho)/2 }] \nonumber \\
&& \leq c_{\rho, 1} \bigg( \mathbf{E} \Big( \sum_{i=1}^n \mathbf{E}[  (\xi_i^2-\mathbf{E}[\xi_i^2 | \mathcal{F}_{i-1}] )^2   |\mathcal{F}_{i-1} ]   \Big)^{ (2+\rho)/4 }    + \sum_{i=1}^n\mathbf{E} | \xi_i^2-\mathbf{E}[\xi_i^2 | \mathcal{F}_{i-1}]| ^{(2+\rho)/2}   \bigg)\nonumber \\
&& \leq     c_{\rho, 2} \bigg( \mathbf{E} \Big( \sum_{i=1}^n \mathbf{E}[  \xi_i^4     |\mathcal{F}_{i-1} ]   \Big)^{ (2+\rho)/4 }   + \sum_{i=1}^n\mathbf{E} | \xi_i |^{ 2+\rho }   \bigg)\nonumber \\
&& \leq     c_{\rho, 3} \, \big(  \varepsilon_n ^{ (2+\rho)/2 }     + \varepsilon_n^\rho   \big)  .
\end{eqnarray}
This completes the proof of the lemma.
\end{proof}

Consider the predictable process $\Psi (\lambda )=(\Psi
_k(\lambda ),\mathcal{F}_k)_{k=0,...,n},$ which is related to the martingale $M$ as follows:
\begin{equation}
\Psi _k(\lambda )=\sum_{i=1}^k\ln \mathbf{E}[e^{\zeta_i(\lambda)}|\mathcal{F}_{i-1}]. \label{f29}
\end{equation}
By equality (\ref{fsd2}), we easily obtain  the following elementary bound for the process $\Psi
(\lambda ).$
\begin{lemma}
\label{lem3.3} Assume  conditions (A2) and (A3) with $\rho \in (0 , 1]$. Then  for all $0 \leq \lambda =o( \min\{\varepsilon_n^{-1}, \kappa_n^{-1}\})  ,$
\[
\Psi _n(\lambda ) =   O(1) \lambda^{2+\rho}  \varepsilon_n^\rho   ,
\]
where $O(1)$ is bounded by an absolute constant.
\end{lemma}

In the proofs of Theorems \ref{th1} and \ref{th2}, we make use of the following assertion, which gives
us a rate of convergence in the CLT for the conjugate
martingale $Y( \lambda  )$ under the probability measure $\mathbf{P}_{ \lambda  }.$
\begin{proposition}
\label{lem3.4}
Assume    conditions (A1) and (A4).

\begin{description}
   \item[\textbf{[i]}] If $\rho \in (0, 1)$, then for all $0\leq \lambda =o(\gamma_n^{-1})$,
\[
\sup_{x}\Big| \mathbf{P}_\lambda (  Y_n(\lambda )/\lambda  \leq x)-\Phi (x)\Big| \leq
c_{  \rho} \, \Big( \lambda^\rho \gamma_n^\rho+  \gamma_n^\rho       + \delta_n \Big) ;
\]
   \item[\textbf{[ii]}] If $\rho =1$, then  for all $0\leq \lambda =o(\gamma_n^{-1})$,
\[
\sup_{x}\Big| \mathbf{P}_\lambda (  Y_n(\lambda )/\lambda \leq x)-\Phi (x)\Big| \leq
c  \, \Big(\lambda  \gamma_n+   \gamma_n  |\ln \gamma_n|     +  \delta_n \Big) ;
\]
 \end{description}
with the convention that $Y_n(0 )/0 =\sum_{i=1}^n\xi_i$.
\end{proposition}

Similarly, we have the following Berry-Esseen bound.
\begin{proposition}
\label{lem3.5}
Assume   conditions  (A1), (A2), and (A3). Then for all $0\leq \lambda =o(\max\{ \varepsilon_n^{-1},  \kappa_n^{-1}  \})$,
\[
\sup_{x}\Big| \mathbf{P}_\lambda (  Y_n(\lambda )/\lambda  \leq x)-\Phi (x)\Big| \leq
c_{  \rho} \, \Big( \lambda^{\rho/2} \gamma_n^{\rho/2}+  \varepsilon_n^{\rho/(3+\rho)}  + \delta_n \Big),
\]
with the convention that $Y_n(0 )/0 =\sum_{i=1}^n\xi_i$.
\end{proposition}
The proofs of Propositions \ref{lem3.4} and  \ref{lem3.5} are much more complicated and we give details in the supplemental article \citep{FGLS17}.

\section{\textbf{Proof of the main results}}\label{sec5}
\setcounter{equation}{0}
We start with the proofs of Theorems \ref{th1} and \ref{th2}, and conclude with the proof of Theorem \ref{th3.1}.
\subsection{Proof of Theorem  \ref{th1} }
By (\ref{denow}), it is easy to see that
$$\Big\{S_n \geq x\sqrt{[S]_n}\Big\}=\Big\{M_n \geq x\sqrt{[M]_n}\Big\}  \supseteq \Big\{ M_n  \geq \frac{x^2 + \lambda^2 [M]_n  }{2 \lambda}  \Big\}=\Big\{ \sum_{i=1}^n\zeta_i(\lambda)  \geq \frac{x^2}{2}  \Big\}. $$
For all $0\leq \lambda   = o(\gamma_n^{-1}),$ according to (\ref{f21}), (\ref{xb}) and (\ref{f29}), we have the following representation:
\begin{eqnarray}
\mathbf{P}\Big(W_n  \geq x  \Big)
&=& \mathbf{E}_\lambda \Big[ Z_n (\lambda)^{-1}\mathbf{1}_{\{ S_n \geq x\sqrt{[S]_n}  \}} \Big]  \nonumber\\
&=& \mathbf{E}_\lambda\Big[ \exp \Big\{
-   \sum_{i=1}^n \zeta_i(\lambda)    +\Psi _n(\lambda )\Big\} \mathbf{1}_{\{ M_n \geq x \sqrt{[M]_n}    \}}\Big] \nonumber\\
&\geq& \mathbf{E}_\lambda\Big[ \exp \Big\{
- Y_n(\lambda ) -  B_n(\lambda)  +\Psi _n(\lambda )\Big\} \mathbf{1}_{\{ \sum_{i=1}^n\zeta_i(\lambda)  \geq \frac{x^2}{2}    \}}\Big]   \nonumber\\
&=&  \mathbf{E}_\lambda\Big[ \exp \Big\{
- Y_n(\lambda ) -  B_n(\lambda)   +\Psi _n(\lambda )\Big\}  \mathbf{1}_{\{ Y_n(\lambda )   \geq \frac{x^2}{2} - B_n(\lambda) \}}\Big]  . \nonumber
\end{eqnarray}
Using Lemmas  \ref{lem3.1},  \ref{lem3.2}  and  \ref{lem3.3}, we get
\begin{eqnarray}
\mathbf{P}\Big( W_n  \geq x  \Big)
   &\geq&   \mathbf{E}_\lambda\Big[ \exp \Big\{
- Y_n(\lambda ) - \Big(\frac12 \lambda^2\langle M \rangle_n + c_1 \lambda^{2+\rho}  \varepsilon_n^\rho \Big)\Big\} \nonumber\\
&& \ \  \ \ \ \ \    \times \mathbf{1}_{\{ Y_n(\lambda )  \geq \frac{x^2}{2} - (\frac12 \lambda^2\langle M \rangle_n + c_1 \lambda^{2+\rho}  \varepsilon_n^\rho ) \}}\Big]  . \nonumber
\end{eqnarray}
Condition (A1) implies that
$$|\langle M \rangle_n -1 | \leq \delta_n^2,$$
and thus
\begin{eqnarray}
 \mathbf{P}\Big(W_n \geq x \Big)
& \geq& \ \mathbf{E}_\lambda\Big[ \exp \Big\{
- Y_n(\lambda ) - \Big(\frac12 \lambda^2 + c_1 \lambda^{2+\rho}  \varepsilon_n^\rho \Big) (1+\delta_n^2)\Big\} \nonumber\\
&& \ \  \ \ \ \ \ \ \   \times \mathbf{1}_{\{ Y_n(\lambda )   \geq \frac{x^2}{2} - (\frac12 \lambda^2 + c_1 \lambda^{2+\rho}  \varepsilon_n^\rho ) (1+\delta_n^2) \}}\Big]  . \label{th2f32}
\end{eqnarray}
Let
$\overline{\lambda }=\overline{\lambda }(x)$ be the largest
solution of the following equation
\[
\Big(\frac12 \lambda^2 + c_1 \lambda^{2+\rho}  \varepsilon_n^\rho \Big ) (1+\delta_n^2)=\frac{x^2}{2} .
\]
The definition of $\overline{\lambda }$ implies that  for all $0 \leq x =o(\gamma_n^{-1}) ,$
\begin{equation}\label{f34}
c_{2}\,x \leq \overline{\lambda }\leq \frac{x } { \sqrt{1+\delta_n^2}}
\end{equation}
and
\begin{equation}\label{f35}
\overline{\lambda }=x+c_{3}\theta_{0}(x^{1+\rho}\varepsilon_n^\rho +x \delta_n ^2),
\end{equation}
where $0\leq \theta_{0} \leq 1.$
From (\ref{th2f32}), we obtain
\begin{equation}
\mathbf{P}\Big(W_n  \geq x  \Big) \geq  \exp\bigg\{ - \Big (\frac12 \overline{\lambda}^2 + c_1 \overline{\lambda}^{2+\rho}  \varepsilon_n^\rho \Big) (1+\delta_n^2) \bigg\}\mathbf{E}_{\overline{\lambda }}\Big[e^{- Y_n(\overline{\lambda })}\mathbf{1}_{\{ Y_n(\overline{\lambda }) \geq 0  \}} \Big] .  \label{ines36df}
\end{equation}
Setting $F_n(y)=\mathbf{P}_{\overline{\lambda }}(Y_n(\overline{\lambda })\leq y ),$
we get
\begin{equation}\label{f37}
\mathbf{P}\Big( W_n   \geq x  \Big)\geq \exp\bigg\{ -c_4\big( \overline{\lambda}^2\delta_n^2+\overline{\lambda}^{2+\rho}  \varepsilon_n^\rho\big)- \frac{\overline{\lambda }^2}{2} \bigg\}\int_0^\infty e^{-
 y}dF_n(y) .
\end{equation}
By integration by parts, we have the following bound:
\begin{equation}\label{f3l8}
\int_0^\infty e^{- y}dF_n(y)\geq\int_0^\infty e^{- y}d\Phi (y/\overline{\lambda})-2\sup_y\Big| F_n(y)-\Phi (y/\overline{\lambda})\Big| .
\end{equation}
For $\rho$,   we distinguish two cases as follows.

\emph{Case 1}: If $\rho \in (0, 1)$, combining (\ref{f37}) and (\ref{f3l8}),  by Proposition \ref{lem3.4}, we have for all
$0 \leq  x  =o(\gamma_n ^{-1}) ,$
\begin{eqnarray}
\mathbf{P}\Big(W_n   \geq x  \Big) &\geq&     \exp\bigg\{ -c_4 \big( \overline{\lambda}^2\delta_n^2+\overline{\lambda}^{2+\rho}  \varepsilon_n^\rho \big) - \frac{\overline{\lambda }^2}{2} \bigg\} \ \ \ \ \ \ \  \  \ \  \  \ \
  \nonumber \\ &&  \, \times\left( \int_0^\infty e^{-
\overline{\lambda }y}d\Phi (y)  - c_{1, \rho}\left( \overline{\lambda}^\rho \gamma_n^\rho +   \gamma_n^\rho       + \delta_n \right)  \right). \label{sfd1s}
\end{eqnarray}
Because
\begin{equation}\label{fsphi}
e^{- \lambda^2/2}\int_0^\infty e^{- \lambda  y}d\Phi
(y)=1-\Phi \left(  \lambda \right)
\end{equation}
and
\begin{equation}
\frac 1{ 1+\lambda  }e^{- \lambda^2/2}\leq   \sqrt{2 \pi}  \Big(
1-\Phi \left(  \lambda \right) \Big) , \ \ \ \ \ \lambda \geq 0, \label{f39}
\end{equation}
we obtain the following lower bound
\begin{eqnarray}
\frac{\mathbf{P}\big(W_n  \geq x  \big)}{1-\Phi \left( \overline{\lambda }\right)} &\geq&   \exp\left\{ -c_4(\overline{\lambda}^2\delta_n^2+ \overline{\lambda}^{2+\rho}  \varepsilon_n^\rho ) \right\} \Big( 1-c_{2, \rho}\, (1+ \overline{\lambda})(\overline{\lambda}^\rho \gamma_n^\rho +\gamma_n^\rho    +  \delta_n  \, ) \Big) \nonumber \\
 &\geq& \exp\left\{ -c_{3, \rho}\Big(\overline{\lambda}^2\delta_n^2+ \overline{\lambda}^{2+\rho}  \varepsilon_n^\rho + (1+ \overline{\lambda})(\overline{\lambda}^\rho \gamma_n^\rho  +\gamma_n^\rho    +  \delta_n  \, ) \Big)\right\}  ,\label{f40}
\end{eqnarray}
for all $0\leq \overline{\lambda} \leq \frac1{2 c_{2, \rho}}\min\{ \gamma_n^{-\rho/(1+\rho)} , \delta_n^{-1} \}.$

Next, we consider the case of $\frac1{2 c_{2, \rho}} \min\{ \gamma_n^{-\rho/(1+\rho)} , \delta_n^{-1} \} \leq \overline{\lambda} =o(\gamma_n^{-1}).$ Let $K \geq 1$ be an absolute constant, whose exact value is chosen later.
It is easy to see that
\begin{eqnarray}\label{jknjsta}
\mathbf{E}_{\overline{\lambda}} \left[e^{- Y_n(\overline{\lambda})}\mathbf{1}_{\{ Y_n(\overline{\lambda})\geq0 \}} \right] &\geq& \mathbf{E}_{\overline{\lambda}} \Big[e^{- Y_n(\overline{\lambda})}\mathbf{1}_{\{0\leq Y_n(\overline{\lambda})\leq \overline{\lambda} K \tau \}} \Big] \nonumber\\
 &\geq&e^{-\overline{\lambda} K \tau}\mathbf{P}_{\overline{\lambda}} \Big(0\leq Y_n(\overline{\lambda})\leq \overline{\lambda} K \tau \Big),
\end{eqnarray}
where $\tau = \overline{\lambda}^{\rho }  \gamma_n^{\rho }       + \delta_n . $
By  Proposition \ref{lem3.4}, we have
\begin{eqnarray*}
\mathbf{P}_{\overline{\lambda}} \Big(0\leq Y_n(\overline{\lambda})\leq    \overline{\lambda} K \tau \Big) &\geq&   \mathbf{P}  \Big( 0\leq  \mathcal{N}(0, 1)
\leq  K \tau  \Big)  -  c_{4, \rho} \tau  \\
  &\geq& \frac{1}{\sqrt{2\pi}} K \tau  e^{-K^2  \tau^2/2}   - c_{4, \rho} \tau\\
  &\geq& \Big( \frac1{3} K    - c_{4, \rho}\Big) \tau.
\end{eqnarray*}
Letting $K\geq   12c_{4,\rho} $, it follows that
$$\mathbf{P}_{\overline{\lambda}} \Big(0\leq Y_n(\overline{\lambda })\leq  \overline{\lambda} K \tau \Big)  \geq \frac1{4} K \tau = \frac1{4} K  \frac{ \overline{\lambda}^{1+\rho }  \gamma_n^{\rho }  + \overline{\lambda }\delta_n} { \overline{\lambda } }.$$
Choosing $$K=\max \Big\{12c_{4, \rho },  \frac{4 }{ \sqrt{\pi}}(2c_{2, \rho})^{1+\rho} \Big\}$$ and taking into account that
$ \frac1{2 c_{2, \rho}} \min\{ \gamma_n^{-\rho/(1+\rho)} , \delta^{-1} \} \leq \overline{\lambda} =o(\gamma_n^{-1})$,
we conclude that
\begin{eqnarray*}
\mathbf{P}_{\overline{\lambda }} \Big(0\leq Y_n(\overline{\lambda }) \leq \overline{\lambda} K \tau \Big)  \geq    \frac{1}{\sqrt{\pi}\overline{\lambda} } .
\label{jknjstb}
\end{eqnarray*}
Because the inequality
$\frac{1}{\sqrt{\pi}\lambda }  e^{- \lambda^2/2}  \geq 1-\Phi \left(  \lambda \right)$ is valid for all $\lambda \geq 1$,
it follows that  for all $\frac1{2 c_{2, \rho}} \min\{ \gamma_n^{-\rho/(1+\rho)} , \delta^{-1} \} \leq \overline{\lambda} =o(\gamma_n^{-1})$,
\begin{eqnarray}\label{dfac}
\mathbf{P}_{\overline{\lambda}} \Big(0\leq Y_n(\overline{\lambda})\leq  K \tau\Big)   \geq \Big( 1-\Phi \left(  \overline{\lambda} \right) \Big)e^{\overline{\lambda}^2/2} .
\end{eqnarray}
Combining (\ref{ines36df}), (\ref{jknjsta}), and (\ref{dfac}), we obtain
 \begin{eqnarray}
\frac{\mathbf{P}\big(W_n \geq x \big)}{1-\Phi \left( \overline{\lambda}\right)} \geq  \exp \bigg \{ -c_{5, \rho} \left( \overline{\lambda}^2\delta_n^2+ \overline{\lambda}^{2+\rho}  \varepsilon_n^\rho + (1+ \overline{\lambda})(\overline{\lambda}^\rho \gamma_n^\rho + \gamma_n^\rho    +  \delta_n  \, )\right)\bigg\},  \label{fgj53}
\end{eqnarray}
which is valid for all $ \frac1{2 c_{2, \rho}} \min\{ \gamma_n^{-\rho/(1+\rho)},  \delta^{-1} \} \leq \overline{\lambda}  =o(\gamma_n^{-1})$.

From (\ref{f40}) and (\ref{fgj53}), we get for all $ 0 \leq \overline{\lambda}   =o(\gamma_n^{-1})$,
\begin{eqnarray}
\frac{\mathbf{P}\big(W_n \geq x  \big)}{1-\Phi \left( \overline{\lambda}\right)} \geq  \exp \bigg \{ -c_{6, \rho} \left( \overline{\lambda}^2\delta_n^2+ \overline{\lambda}^{2+\rho}  \varepsilon_n^\rho + (1+ \overline{\lambda})(\overline{\lambda}^\rho \gamma_n^\rho + \gamma_n^\rho    +  \delta_n  \, )\right)\bigg\}.  \label{fgjfd36}
\end{eqnarray}
Next, we   substitute $x$ for $\overline{\lambda }$ in the tail
of the normal law $1-\Phi (\overline{\lambda }).$
By (\ref{f34}), (\ref{f35}), and (\ref{f39}), we get
\begin{eqnarray}
   1 \leq \frac{\int_{\overline{\lambda}}^{ \infty}\exp\{- t^2/2 \}d t}{\int_{x}^{ \infty}\exp\{- t^2/2 \}d t}   &\leq &
  1+\frac{\int_{\overline{\lambda}}^{x}\exp\{ -t^2/2 \} d t}{\int_{x}^{ \infty}\exp\{-t^2/2\}d t}\nonumber\\
   & \leq & 1+c_{1}x(x-\overline{\lambda}) \exp\Big\{ (x^2-\overline{\lambda}^2)/2 \Big\}\nonumber\\
   & \leq & \exp\Big\{ c_2\, (x^2\delta_n^2+ x^{2+\rho}  \varepsilon_n^\rho)\Big\} \label{f41}
\end{eqnarray}
and hence
\begin{equation}\label{f42}
1-\Phi \left( \overline{\lambda }\right) =\left( 1-\Phi (x)\right)\exp \left\{ \theta_{1} c\, (
x^{2+\rho}  \varepsilon_n^\rho+ x^2 \delta_n^2 ) \right\}.
\end{equation}
Implementing (\ref{f42}) in (\ref{fgjfd36}) and using (\ref
{f34}), we obtain  for all $0\leq x =o(\gamma_n ^{-1}) ,$
\begin{eqnarray*}
\frac{\mathbf{P}\big(W_n   \geq x  \big)}{1-\Phi \left( x\right) } \geq  \exp\bigg\{ -c_{7, \rho} \Big( x^{2+\rho}  \varepsilon_n^\rho+ x^2 \delta_n^2 +(1+x)(
 x^\rho\gamma_n^\rho  + \gamma_n^\rho      +  \delta_n)\Big) \bigg \}  \nonumber,
\end{eqnarray*}
which gives the desired  lower bound (\ref{t1ie1}).

\emph{Case 2}: If $\rho =1,$  using Proposition  \ref{lem3.4} with $\rho=1,$ we have  for all $0 \leq  x  =o(\gamma_n ^{-1}) ,$
\begin{eqnarray*}
 \mathbf{P}\Big(W_n\geq x  \Big) &\geq&     \exp\bigg\{ -c_1\big( \overline{\lambda}^2\delta_n^2+\overline{\lambda}^{3}  \varepsilon_n  \big) - \frac{\overline{\lambda }^2}{2} \bigg\} \ \ \ \ \ \ \  \  \ \  \  \ \
  \nonumber \\ && \  \times \left( \int_0^\infty e^{-
\overline{\lambda }y}d\Phi (y)  - c_2\Big(  \overline{\lambda} \gamma_n+  \gamma_n |\ln \gamma_n|      + \delta_n \Big)  \right),
\end{eqnarray*}
that is, the term $\gamma_n^\rho$ in inequality (\ref{sfd1s}) has been replaced by $\gamma_n |\ln \gamma_n|.$
By an argument similar to that of  \emph{Case 1},  we obtain the desired lower bound (\ref{t1ie12}).

Notice that $(-S_k, \mathcal{F}_k)_{k=0,...,n}$  also satisfies   conditions (A1), (A2), and (A4).  Thus, the same inequalities hold  when   $\frac{\mathbf{P}(W_n\geq x)}{1-\Phi \left( x\right)}$ is replaced by $\frac{\mathbf{P}(W_n\leq-x)}{ \Phi \left( -x\right)}$ for all $0 \leq  x  =o(\gamma_n ^{-1})$. This completes the proof of
 Theorem \ref{th1}.  \hfill\qed

\subsection{Proof of Theorem  \ref{th2}}

We first prove Theorem \ref{th2} for all $ 1\leq x    = o(\gamma_n^{-1}).$
Observe that
\begin{eqnarray}
\mathbf{P}\Big(W_n \geq x  \Big)&=& \mathbf{P}\Big(W_n \geq x ,\, | [M]_n - \langle M\rangle_n |\leq  \delta_n+  1/(2x) \Big)  \nonumber \\
&&+\ \mathbf{P}\Big(W_n \geq x ,\, | [M]_n - \langle M\rangle_n |> \delta_n+  1/(2x) \Big) . \label{twoterm}
\end{eqnarray}
For the the first term on the right hand side of (\ref{twoterm}), by (\ref{f21}) and (\ref{f23}) with $\lambda=x$, we have the following representation:
\begin{eqnarray}
 &&\mathbf{P}\Big(W_n \geq x ,\, | [M]_n - \langle M\rangle_n |\leq  \delta_n+  1/(2x) \Big) \nonumber \\
&& = \mathbf{E}_x \Big[ Z_n (x)^{-1}\mathbf{1}_{\{ M_n \geq x\sqrt{[M]_n},\  | [M]_n - \langle M\rangle_n   | \leq \delta_n+ 1/(2x)  \}} \Big]  \nonumber\\
&& =\mathbf{E}_x\Big[ e^{
-  Y_n(x) -  B_n(x)   +\Psi _n(x) } \mathbf{1}_{\big\{ x M_n \geq   x^2 \sqrt{ 1 +  [M]_n - 1 },\   | [M]_n - \langle M\rangle_n   | \leq \delta_n+1/(2x)  \big \}}\Big] .\nonumber
\end{eqnarray}
By the inequality $$ \sqrt{1+ y} \geq 1+ y/2 -y^2/2, \ \ \ \ \ \   y\geq -1,$$
condition (A1) and Lemma \ref{lem3.2},  we have for all $ 1\leq x    = o(\gamma_n^{-1}),$
\begin{eqnarray}
&&  \mathbf{P}\Big(W_n\geq x\,,\  | [M]_n - \langle M\rangle_n  | \leq \delta_n+  1/(2x) \Big) \nonumber \\
&&\leq \mathbf{E}_x\Big[ \exp \Big\{
- Y_n(x) -  B_n(x)   +\Psi _n(x )\Big\} \nonumber\\
&& \ \ \ \ \ \ \ \    \times \mathbf{1}_{\big\{ x M_n - \frac12  x^2 [M]_n    +  \frac12 x^2([M]_n - 1)^2   \geq   \frac12 x^2 , \  | [M]_n - \langle M\rangle_n   | \leq \delta_n+ 1/(2x)     \big\}}\Big]    \nonumber \\
&&\leq \mathbf{E}_x\Big[ \exp \Big\{
- Y_n(x) -  B_n(x)   +\Psi _n(x )\Big\} \nonumber\\
&& \ \ \ \ \ \  \ \  \times \mathbf{1}_{\big\{ x M_n - \frac12  x^2 [M]_n    +    x^2([M]_n -\langle M\rangle_n  )^2  + x^2(1 - \langle M\rangle_n )^2  \geq   \frac12 x^2 , \  | [M]_n - \langle M\rangle_n   | \leq \delta_n+ 1/(2x)       \big\}}\Big]    \nonumber \\
&&\leq \mathbf{E}_x\Big[ \exp \Big\{
- Y_n(x) -  B_n(x)   +\Psi _n(x )\Big\} \nonumber\\
&& \ \ \ \ \ \  \ \  \times \mathbf{1}_{\big\{ Y_n(x)  \geq  -   x^2([M]_n -\langle M\rangle_n  )^2 - x^2\delta_n^4 + \frac12 x^2 - B_n(x) , \  | [M]_n - \langle M\rangle_n   | \leq  \delta_n+1/(2x)       \big\}}\Big]    \nonumber \\
&&\leq \mathbf{E}_x\Big[ \exp \Big\{
- Y_n(x) -  B_n(x)   +\Psi _n(x )\Big\} \nonumber\\
&& \ \ \ \ \ \ \ \   \times \mathbf{1}_{\big\{ Y_n(x)   \geq  -  x^{2+\rho}  \varepsilon_n^\rho - x^2\delta_n^4 + \frac12 x^2 - B_n(x) , \  | [M]_n - \langle M\rangle_n   | \leq  (x \varepsilon_n)^{\rho/2}     \big\}}\Big] \nonumber \\
&&\ \ \  +\, \mathbf{E}_x\Big[ \exp \Big\{
- Y_n(x) -  B_n(x)   +\Psi _n(x )\Big\} \nonumber\\
&& \ \ \ \ \ \ \ \    \times \mathbf{1}_{\big\{0>  Y_n(x)  \geq  -   x^2([M]_n -\langle M\rangle_n  )^2 - x^2\delta_n^4 + \frac12 x^2 - B_n(x) , \  (x \varepsilon_n)^{\rho/2} <  | [M]_n - \langle M\rangle_n   | \leq \delta_n+ 1/(2x)     \big\}}\Big]  \nonumber  \\
&&\leq \mathbf{E}_x\Big[ \exp \Big\{
- Y_n(x)-  B_n(x)   +\Psi _n(x )\Big\} \nonumber\\
&& \ \ \ \ \ \ \ \   \times \mathbf{1}_{\big\{ Y_n(x)    \geq  - c_1( x^{2+\rho}  \varepsilon_n^\rho + x^2\delta_n^2)    \big\}}\Big] \nonumber \\
&&\ \ \  +\, \mathbf{E}_x\Big[ \exp \Big\{
- Y_n(x) -  B_n(x)   +\Psi _n(x )\Big\} \nonumber\\
&& \ \ \ \ \ \ \ \    \times \mathbf{1}_{\big\{0>  Y_n(x)  \geq  - \frac14 - c_2( x^{2+\rho}  \varepsilon_n^\rho + x^2\delta_n^2) , \  (x \varepsilon_n)^{\rho/2} <  | [M]_n - \langle M\rangle_n   | \leq \delta_n+ 1/(2x)     \big\}}\Big]  \nonumber  \\
&&:= I_1(x)+ I_2(x) .  \label{ineq62}
\end{eqnarray}
For $I_1(x),$ by an argument similar to the proof of  Theorem \ref{th1}, we get for all $0\leq x =o(\gamma_n^{-1})$,
\begin{eqnarray}
 &&\!\!\!\!\!\!\!\! \frac{I_1(x)}{1-\Phi \left( x\right)}
 \leq\left\{ \begin{array}{ll}
 \!\!\exp\bigg\{  c_{\rho  } \Big( x^{2+\rho}  \varepsilon_n^\rho+ x^2 \delta_n^2 +(1+x)\left(  x^\rho \gamma_n^\rho +   \gamma_n^\rho      +  \delta_n   \right) \Big) \bigg \} &   \textrm{if $\rho \in (0, 1),$}\\
  & \\
  \exp\bigg\{  c  \Big( x^{3}  \varepsilon_n + x^2 \delta_n^2+ (1+x)\left(x \gamma_n+  \gamma_n|\ln \gamma_n|     +  \delta_n   \right) \Big) \bigg \} &   \textrm{if $\rho =1.$}
\end{array} \right. \nonumber\\
&&     \label{ineq45}
\end{eqnarray}
Next,   consider the item $I_2(x).$  By condition (A1),  Lemmas \ref{lem3.2} and \ref{lem3.3},   it is obvious that for all $ 1\leq x    = o(\gamma_n^{-1}),$
\begin{eqnarray}
\!\!\!I_2(x)\!\!&\leq&\!\!\exp\bigg\{ -\frac12 x^2+  c_1  \big( x^{2+\rho}  \varepsilon_n^\rho+ x^2 \delta_n^2   \big) \bigg \} \nonumber \\
&& \times \mathbf{E}_x\Big[ e^ {
- Y_n(x)  }\mathbf{1}_{\big\{0>  Y_n(x)  \geq  -  \frac14 - c_2  ( x^{2+\rho}  \varepsilon_n^\rho+ x^2 \delta_n^2  ) , \ (x \varepsilon_n)^{\rho/2} < | [M]_n - \langle M\rangle_n   |      \big\}}\Big] \nonumber\\
&\leq&\!\!\exp\bigg\{  -\frac12 x^2+  c_1  \big( x^{2+\rho}  \varepsilon_n^\rho+ x^2 \delta_n^2   \big) \bigg \} \nonumber \\
&& \times \mathbf{E}_x\Big[ e^ {\frac14 +c_2  ( x^{2+\rho}  \varepsilon_n^\rho+ x^2 \delta_n^2  )  }\mathbf{1}_{\big\{ (x \varepsilon_n)^{\rho/2} <| [M]_n - \langle M\rangle_n   |      \big\}}\Big] \nonumber\\
&\leq&\!\! e^{\frac14} \exp\bigg\{  -\frac12 x^2+  c_3  \big( x^{2+\rho}  \varepsilon_n^\rho+ x^2 \delta_n^2   \big) \bigg \}   \mathbf{E}_x\Big[  \mathbf{1}_{\big\{(x \varepsilon_n)^{\rho/2} < | [M]_n - \langle M\rangle_n   |      \big\}}\Big]. \label{gjmsg02}
\end{eqnarray}
Denote by  $\langle M(x) \rangle_n = \sum_{i=1}^n\mathbf{E}_x[\xi_i^2 | \mathcal{F}_{i-1}].$ Notice that $\varepsilon_n = O(1)\gamma_n.$
From (\ref{sfghkm423}), using  (\ref{fsd2}), Lemmas \ref{lem003}, \ref{lem3.1} and condition (A2),   we obtain for all $ 1\leq x    = o(\gamma_n^{-1}),$
\begin{eqnarray}\label{f56}
 &&\Big |\langle M(x) \rangle_n - \langle M  \rangle_n   \Big | \nonumber \\
  &&\leq \sum_{i=1}^n\bigg| \frac{\mathbf{E}[\xi
_i^2e^{x \xi_i - x^2 \xi_i^2/2}|\mathcal{F}_{i-1}]}{\mathbf{E}[e^{x \xi_i - x^2 \xi_i^2/2}|\mathcal{F}%
_{i-1}]} - \mathbf{E}[\xi_i^2 |\mathcal{F}_{i-1}]\bigg|  +\sum_{i=1}^n \bigg(\frac{\mathbf{E}[\xi _ie^{x \xi_i - x^2 \xi_i^2/2}|\mathcal{F}_{i-1}]^2}{\mathbf{E}[e^{x \xi_i - x^2 \xi_i^2/2}|\mathcal{F}_{i-1}]^2}\bigg)  \nonumber \\
&&\leq c_1 \sum_{i=1}^n \Big(  \mathbf{E}[x^\rho |\xi _i|^{2+\rho} | \mathcal{F}_{i-1}]  +   (   \mathbf{E}[x \xi _i^2 | \mathcal{F}_{i-1}])^2  \Big)   \nonumber\\
  &&\leq c_1 \sum_{i=1}^n \Big(  \mathbf{E}[x^\rho |\xi _i|^{2+\rho} | \mathcal{F}_{i-1}]  + x^{ 2} \mathbf{E}[ |\xi _i|^{2+\rho} | \mathcal{F}_{i-1}]  (   \mathbf{E}[ \xi _i^2 | \mathcal{F}_{i-1}])^{ (2-\rho)/2}  \Big)   \nonumber\\
&&\leq c_2\, x^\rho \varepsilon_n^\rho. \label{sec8fs7}
\end{eqnarray}
Thus, for all $1\leq x =  o(\gamma_n^{-1}),$
\begin{eqnarray}
I_2(x)
&\leq& e^{\frac14} \exp\bigg\{  -\frac12 x^2+  c_3  \big( x^{2+\rho}  \varepsilon_n^\rho+ x^2 \delta_n^2   \big) \bigg \}   \mathbf{E}_x\Big[  \mathbf{1}_{\big\{ \frac12(x \varepsilon_n)^{\rho/2} < | [M]_n - \langle M(x) \rangle_n   |     \big\}}\Big] \nonumber\\
&\leq&  \frac{4 e^{\frac14} }{(x \varepsilon_n)^{  \rho(2+\rho)/4 }} \exp\bigg\{  -\frac12 x^2+  c_3  \big( x^{2+\rho}  \varepsilon_n^\rho+ x^2 \delta_n^2   \big) \bigg \}   \mathbf{E}_x\big[ | [M]_n - \langle M(x) \rangle_n |^{(2+\rho)/2 }    \big]. \nonumber
\end{eqnarray}
It is obvious  that
\begin{eqnarray*}
   [M]_n - \langle M(x)\rangle_n      =     \sum_{i=1}^n(\xi_i^2-\mathbf{E}_x[\xi_i^2 | \mathcal{F}_{i-1}]) .
\end{eqnarray*}
Thus, $([M]_i - \langle M(x)\rangle_i, \mathcal{F}_{i})_{i=0,...,n} $ is a martingale with respect to  the   probability measure $\mathbf{P}_x.$
By the inequality  of \citep{V65}, it follows that  for all $1\leq x =  o(\gamma_n^{-1}),$
\begin{eqnarray}
\mathbf{E}_x[ |[M]_n - \langle M(x)\rangle_n |^{(2+\rho)/2} ]   & \leq& c_1 \sum_{i=1}^n\mathbf{E}_x[| \xi_i^2-\mathbf{E}_x[\xi_i^2 | \mathcal{F}_{i-1}]| ^{(2+\rho)/2}] \nonumber \\
& \leq&  c_2 \sum_{i=1}^n\mathbf{E}_x[|\xi_i| ^{ 2+\rho }  ]  \nonumber\\
& =&  c_2 \sum_{i=1}^n  \frac{\mathbf{E}[ |\xi_i| ^{ 2+\rho } e^{\zeta_i(x)}|\mathcal{F}_{i-1}]}{\mathbf{E}[e^{\zeta_i(x)}|\mathcal{F}%
_{i-1}]}  \nonumber\\
& \leq&  c_3 \varepsilon_n^\rho   . \label{ineds65}
\end{eqnarray}
Hence,  for all $1\leq x =  o(\gamma_n^{-1}),$
\begin{eqnarray}
I_2(x) &\leq& C\, \frac{ \, \varepsilon_n^{ \rho(2-\rho)/4 } }{ x  ^{  \rho(2+\rho)/4 }} \exp\bigg\{  -\frac12 x^2+  c_3  \big( x^{2+\rho}  \varepsilon_n^\rho+ x^2 \delta_n^2   \big) \bigg \}. \label{ineq65}
\end{eqnarray}

Next, we give an estimation for $\mathbf{P}\Big(W_n \geq x ,\, | [M]_n - \langle M\rangle_n |> \delta_n+ 1/(2x) \Big).$
It is obvious that
\begin{eqnarray*}
&&\mathbf{P}\Big(W_n \geq x ,\, | [M]_n - \langle M\rangle_n |> \delta_n+ 1/(2x) \Big) \\
 &&\leq \mathbf{P}\Big(W_n \geq x ,\, | [M]_n - 1 | + | 1 - \langle M\rangle_n |> \delta_n+ 1/(2x) \Big) \\
&& \leq\mathbf{P}\Big(W_n \geq x ,\, | [M]_n - 1 |> \delta_n/2 + 1/(2x) \Big).
\end{eqnarray*}
To estimate the tail probability in the last line, we follow the argument of \citep{SZ16}.
We have the following decomposition:
\begin{eqnarray}
&&  \mathbf{P}\Big(W_n \geq x   ,\  | [M]_n - 1  | >  \delta_n/2 + 1/(2x) \Big) \nonumber \\
&&\leq \mathbf{P}\Big(M_n/\sqrt{[M]_n} \geq x,\  1+ \delta_n/2 + 1/(2x) < [M]_n  \leq 16   \Big) \nonumber\\
&& \ \ \ \ \ \   +\, \mathbf{P}\Big(M_n/\sqrt{[M]_n} \geq x,\  [M]_n  < 1 -  \delta_n/2 - 1/(2x) \Big)    \nonumber \\
&& \ \ \ \ \ \  + \, \mathbf{P}\Big(M_n/\sqrt{[M]_n} \geq x,\   [M]_n  >  16 \Big) \nonumber \\
&&:= \sum_{v=1}^3  \mathbf{P}\Big( (M_n,\, \sqrt{[M]_n}) \in \mathcal{E}_v \Big)  , \label{3part}
\end{eqnarray}
where $\mathcal{E}_v\subset \mathbf{R}\times \mathbf{R}^+, 1\leq v \leq 3, $ are given by
\begin{eqnarray}
&&  \mathcal{E}_1 =  \Big\{(u, v)\in \mathbf{R}\times \mathbf{R}^+: u/v \geq x,\,  \sqrt{1 + \delta_n/2 + 1/(2x) } < v \leq 4 \Big \} , \nonumber \\
&&\mathcal{E}_2 = \Big\{(u, v)\in \mathbf{R}\times \mathbf{R}^+: u/v \geq x,\,  v <\sqrt{1 -  \delta_n/2 - 1/(2x) } \ \Big\} ,\nonumber\\
&& \mathcal{E}_3 = \Big\{(u, v)\in \mathbf{R}\times \mathbf{R}^+:  u/v \geq x,\, v > 4 \Big\}. \nonumber
\end{eqnarray}
To estimate the probability $\mathbf{P}((M_n, \sqrt{[M]_n} ) \in \mathcal{E}_1 ), $  we introduce the following new conjugate probability measure $\widetilde{\mathbf{P}}_x$ defined by
$$d\widetilde{\mathbf{P}}_x = \widetilde{Z}_n(x) d\mathbf{P},$$
where
$$\widetilde{Z}_n(x )=\prod_{i=1}^k\frac{e^{\widetilde{\zeta}_i(x)}}{\mathbf{E}[e^{\widetilde{\zeta}_i(x)}|
\mathcal{F}_{i-1}]} \ \ \ \ \textrm{and} \ \ \ \ \widetilde{\zeta}_i(x)=x \xi_i - x^2 \xi_i^2/8.$$
Denote by $\widetilde{\mathbf{E}}_{x}$ the expectation with respect to $\widetilde{\mathbf{P}}_{x}$ and   $\langle \widetilde{M}(x) \rangle_n = \sum_{i=1}^n\widetilde{\mathbf{E}}_x[\xi_i^2 | \mathcal{F}_{i-1}].$  By an argument similar to (\ref{sec8fs7}),
it follows that for all $1\leq x =  o(\gamma_n^{-1}),$
$$ \langle \widetilde{M}(x) \rangle_n = \langle M  \rangle_n  + O(1) x^\rho \varepsilon_n^\rho.  $$
By Markov's inequality, we deduce that
\begin{eqnarray}
 &&\mathbf{P}\Big((M_n, \sqrt{[M]_n} ) \in \mathcal{E}_1  \Big)  \nonumber \\
  &&\leq (\delta_n/2 + 1/(2x))^{-2} e^{ -\inf_{(u,v) \in \mathcal{E}_1} ( xu-(vx)^2/8) } \mathbf{E}[([M]_n-1)^2
e^{  x M_n -[M]_n x^2/8  } ]    \nonumber\\
&& \leq 16x^2 e^{ -\inf_{(u,v) \in \mathcal{E}_1} ( xu-(vx)^2/8) } \mathbf{E}[([M]_n-\langle \widetilde{M}(x) \rangle_n)^2
e^{  x M_n -[M]_n x^2/8  } ]    \nonumber \\
&&\ \ \ \  +\, 16 x^{2} e^{ -\inf_{(u,v) \in \mathcal{E}_1} ( xu-(vx)^2/8) } \mathbf{E}[(\langle \widetilde{M}(x)\rangle_n-\langle M \rangle_n)^2
e^{  x M_n -[M]_n x^2/8  } ] \nonumber \\
&& \ \ \ \  +\, 16 \delta_n^{-2} e^{ -\inf_{(u,v) \in \mathcal{E}_1} ( xu-(vx)^2/8) } \mathbf{E}[(\langle M \rangle_n-1)^2
e^{  x M_n -[M]_n x^2/8  } ] \nonumber \\
&&\leq  16 x^2 e^{ -\inf_{(u,v) \in \mathcal{E}_1} ( xu-(vx)^2/8) } \mathbf{E}[([M]_n-\langle \widetilde{M}(x) \rangle_n)^2
e^{  x M_n -[M]_n x^2/8  } ]    \nonumber \\
&&\ \ \ \    +\, C x^{2+2\rho} \varepsilon_n^{2\rho}   e^{ -\inf_{(u,v) \in \mathcal{E}_1} ( xu-(vx)^2/8) } \mathbf{E}[
e^{  x M_n -[M]_n x^2/8  } ] \nonumber \\
&&\ \ \ \    +\, 16 \delta_n^{ 2} e^{ -\inf_{(u,v) \in \mathcal{E}_1} ( xu-(vx)^2/8) } \mathbf{E}[
e^{  x M_n -[M]_n x^2/8  } ]  , \label{inesf35}
\end{eqnarray}
 where it is easy to verify that
\begin{eqnarray}
\inf_{(u,v) \in \mathcal{E}_1} \Big( xu-\frac18(vx)^2  \Big) \geq  \frac78 x^2 +   \frac14 x   -  c\,   x^2 \delta_n^2  .
\end{eqnarray}
  By Lemma \ref{lem001}, conditions (A1) and (A2), it follows
 that
 \begin{eqnarray*}
 \prod_{i=1}^n \mathbf{E}[e^{\widetilde{\zeta}_i(x)}|
\mathcal{F}_{i-1}]  &\leq& \prod_{i=1}^n \Big ( 1+ \frac38  x^2 \mathbf{E} [ \xi_i^2|
\mathcal{F}_{i-1}]   +O(1)  x^{2+\rho}  \mathbf{E}[ |\xi_i|^{2+\rho}  |  \mathcal{F}_{i-1}]  \Big)\\
&\leq& \prod_{i=1}^n \exp\Big\{\frac38  x^2 \mathbf{E} [ \xi_i^2|
\mathcal{F}_{i-1}]   +O(1)  x^{2+\rho}  \mathbf{E}[ |\xi_i|^{2+\rho}  |  \mathcal{F}_{i-1}]  \Big\}\\
 &=&\exp\Big\{   \frac38  x^2 \langle M \rangle_n   +O(1) x^{2+\rho} \sum_{i=1}^n \mathbf{E}[ |\xi_i|^{2+\rho}  |  \mathcal{F}_{i-1}]  \Big\} \\
  &\leq&\exp\Big\{    \frac38  x^2  +   O(1)  (x^{2+\rho}  \varepsilon_n^\rho+x^2\delta_n^2) \Big\}.
\end{eqnarray*}
Therefore, for all $1\leq x =  o(\gamma_n^{-1}),$
\begin{eqnarray*}
&& \mathbf{E}\Big[([M]_n-\langle \widetilde{M}(x) \rangle_n )^2
e^{  x M_n -[M]_n x^2/8  } \Big] \nonumber \\
&&= \mathbf{E}\Big[ \Big(\Pi_{i=1}^n\mathbf{E}[e^{\widetilde{\zeta}_i(x)}|
\mathcal{F}_{i-1}] \Big)\Big( [M]_n-\langle \widetilde{M}(x) \rangle_n \Big)^2\widetilde{Z}_n(x ) \Big] \\
&&\leq \mathbf{E}\Big[ \Big( [M]_n-\langle \widetilde{M}(x) \rangle_n \Big)^2\widetilde{Z}_n(x ) \Big]\exp\Big\{    \frac38  x^2  +   O(1)  (x^{2+\rho}  \varepsilon_n^\rho+x^2\delta_n^2)  \Big\}\\
&&= \widetilde{\mathbf{E}}_x\Big[ \big( [M]_n-\langle \widetilde{M}(x) \rangle_n \big)^2  \Big]\exp\Big\{    \frac38  x^2  +   O(1) (x^{2+\rho}  \varepsilon_n^\rho+x^2\delta_n^2) \Big\}\\
&&= \sum_{i=1}^n\widetilde{\mathbf{E}}_x\Big[(\xi_i^2-\widetilde{\mathbf{E}}_x [\xi_i^2 | \mathcal{F}_{i-1}])^2 \Big]
\exp\Big\{    \frac38  x^2  +   O(1)  (x^{2+\rho}  \varepsilon_n^\rho+x^2\delta_n^2) \Big\},
\end{eqnarray*}
where the last line follows because  $([M]_i - \langle \widetilde{M}(x) \rangle_i, \mathcal{F}_{i})_{i=0,...,n} $ is a martingale with respect to  the   probability measure $\widetilde{\mathbf{P}}_x.$
Therefore, by Lemma \ref{lem001}, conditions (A1) and (A2) again, we have for all $1\leq x =  o(\gamma_n^{-1}),$
\begin{eqnarray*}
&& \mathbf{E}\Big[([M]_n-\langle \widetilde{M}(x) \rangle_n )^2
e^{  x M_n -[M]_n x^2/8  } \Big] \nonumber \\
&&\leq \sum_{i=1}^n\widetilde{\mathbf{E}}_x\big[ \widetilde{\mathbf{E}}_x\big[ \xi_i^4 |
\mathcal{F}_{i-1}] \big]
\exp\Big\{    \frac38  x^2  +   O(1)  (x^{2+\rho}  \varepsilon_n^\rho+x^2\delta_n^2) \Big\} \nonumber \\
&&= \sum_{i=1}^n\widetilde{\mathbf{E}}_x \Big[    \mathbf{E}[ \xi_i^4  e^{\widetilde{\zeta}_i(x)}|
\mathcal{F}_{i-1}] \Big / \mathbf{E}[e^{\widetilde{\zeta}_i(x)}|
\mathcal{F}_{i-1}]   \Big] \exp\Big\{  \frac38  x^2  +   O(1)  (x^{2+\rho}  \varepsilon_n^\rho+x^2\delta_n^2) \Big\} \\
&&\leq  C_0 \sum_{i=1}^n  \widetilde{\mathbf{E}}_x \Big[   \frac{1 }{x^{2-\rho}} \sum_{i=1}^n \mathbf{E}  [|\xi_i|^{2+\rho}|\mathcal{F}_{i-1} ] \Big]\exp\Big\{  \frac38  x^2  +   O(1) (x^{2+\rho}  \varepsilon_n^\rho+x^2\delta_n^2) \Big\} \\
&&\leq C_1   \varepsilon_n^\rho  \exp\Big\{   \frac38  x^2  +   O(1) (x^{2+\rho}  \varepsilon_n^\rho+x^2\delta_n^2) \Big\}.
\end{eqnarray*}
  Lemma \ref{lem001} implies that for all $1\leq x =  o(\gamma_n^{-1}),$
\begin{eqnarray*}
 && \mathbf{E}\Big[\exp\Big\{ x M_n - \frac18 x^2 [M]_n - \frac38  x^2 \langle M \rangle_n     -O(1) x^{2+\rho}   \sum_{i=1}^n  \mathbf{E}[ |\xi_i|^{2+\rho}  |  \mathcal{F}_{i-1}] \Big\} \Big] \\
&&\leq\mathbf{E}\Big[\exp\Big\{ x M_{n-1} - \frac18 x^2[M]_{n-1} - \frac38  x^2 \langle M \rangle_{n-1}     -O(1) x^{2+\rho}   \sum_{i=1}^{n-1}  \mathbf{E}[ |\xi_i|^{2+\rho}  |  \mathcal{F}_{i-1}] \Big\} \Big] \\
&&\leq 1.
\end{eqnarray*}
By conditions (A1), (A2) and the last inequality, we obtain  for all $1\leq x =  o(\gamma_n^{-1}),$
\begin{eqnarray*}
\mathbf{E}[e^{  x M_n -[M]_n x^2/8  } ]
  \leq \exp\Big\{   \frac38  x^2  +   O(1)   (x^{2+\rho}  \varepsilon_n^\rho+x^2\delta_n^2)  \Big\} .
\end{eqnarray*}
Thus, from (\ref{inesf35}), we deduce that for all $1\leq x =  o(\gamma_n^{-1}),$
\begin{eqnarray}
&&\mathbf{P}\Big((M_n, \sqrt{[M]_n} ) \in \mathcal{E}_1 \Big)  \nonumber \\
&&\leq   C_2  (   \varepsilon_n^\rho + x^{2+2\rho} \varepsilon_n^{2\rho}+ \delta_n^2 ) \exp\Big\{-\frac12 x^2   -\frac14 x  + O(1) (x^{2+\rho}  \varepsilon_n^\rho+x^2\delta_n^2) \Big\} \nonumber \\
&& \leq C_3 (   \varepsilon_n^\rho +   \delta_n^2 ) \exp\Big\{-\frac12 x^2    + O(1) (x^{2+\rho}  \varepsilon_n^\rho+x^2\delta_n^2) \Big\} .\label{3part01}
\end{eqnarray}
Similarly,    we have
\begin{eqnarray}
&&\mathbf{P}\Big((M_n, \sqrt{[M]_n} ) \in \mathcal{E}_2 \Big) \nonumber \\
 & &\leq   (\delta_n/2 + 1/(2x))^{-2} e^{ -\inf_{(u,v) \in \mathcal{E}_2} ( xu-2(vx)^2 ) } \mathbf{E}[([M]_n-1)^2
e^{  x M_n -2[M]_n x^2   } ]    \nonumber\\
&&\leq  C_4 (  \varepsilon_n^\rho+ \delta_n^2 )   \exp\Big\{-\frac12 x^2 + O(1) (x^{2+\rho}  \varepsilon_n^\rho+x^2\delta_n^2) \Big\}. \label{3part02}
\end{eqnarray}
For the last term $\mathbf{P}((M_n, \sqrt{[M]_n} ) \in \mathcal{E}_3 ), $  we obtain the following estimation
\begin{eqnarray}
\mathbf{P}\Big((M_n, \sqrt{[M]_n} ) \in \mathcal{E}_3 \Big)&=&  \mathbf{P}\Big(M_n \geq x\sqrt{[M]_n},\    [M]_n > 16 \Big) \nonumber \\
&\leq& \frac23 x^{-2/3} \exp\Big\{-\frac34 x^2  \Big\},  \label{fsf69}
\end{eqnarray}
where the last line follows by Lemma  \ref{slemms}.
Moreover, by Lemma \ref{slemma58}, it holds that for $\rho \in (0 ,   1],$
\begin{eqnarray*}
\mathbf{P}\Big((M_n, \sqrt{[M]_n} ) \in \mathcal{E}_3 \Big )  &\leq&   \mathbf{P}\Big(  | [M]_n - \langle M\rangle_n  | \geq 1 \Big) \\
&\leq& c\, \varepsilon_n^\rho .
\end{eqnarray*}
By the last inequality  and (\ref{fsf69}), we get for all $1\leq x  =o(\gamma_n^{-1}) $,
\begin{eqnarray}
\mathbf{P}\Big((M_n, \sqrt{[M]_n} ) \in \mathcal{E}_3 \Big)&\leq&  \min\Big\{c\, \varepsilon_n^\rho, \,  \frac23 x^{-2/3} e^{- 3 x^2 /4  }  \Big\}  \nonumber \\
\ \ \ &\leq&  C\, \frac{  \varepsilon_n^{ \rho(2-\rho)/4 } }{ x  ^{  \rho(2+\rho)/4 }} \exp\Big\{-\frac12 x^2  \Big\} .\label{fskl12f}
\end{eqnarray}
Thus, combining  the inequalities (\ref{3part}), (\ref{3part01}), (\ref{3part02}) and (\ref{fskl12f}) together, we deduce that for all $ 1 \leq x =  o(\gamma_n^{-1}),$
\begin{eqnarray}
 && \mathbf{P}\Big(W_n  \geq x ,\, | [M]_n - \langle M\rangle_n |> \delta_n+ 1/(2x) \Big)  \nonumber  \\
  &&  \ \   \leq  C \Big (  \frac{  \varepsilon_n^{ \rho(2-\rho)/4 } }{ x  ^{  \rho(2+\rho)/4 }} + \delta_n^2  \Big) \exp\Big\{-\frac12 x^2 + O(1) (x^{2+\rho}  \varepsilon_n^\rho+x^2\delta_n^2) \Big\}. \label{ineq75}
\end{eqnarray}
Combining  (\ref{ineq62}), (\ref{ineq45}), (\ref{ineq65}), and (\ref{ineq75}),  we obtain  for all $1 \leq x =  o(\gamma_n^{-1}),$
\begin{eqnarray}
&& \frac{\mathbf{P}(W_n\geq x  )}{1-\Phi \left( x\right)}
  \leq \bigg( 1+  C(1+x) \Big(      \frac{  \varepsilon_n^{ \rho(2-\rho)/4 } }{ x  ^{  \rho(2+\rho)/4 }}
+\delta_n^2    \Big) \bigg)   \nonumber \\
& &\ \ \ \ \ \ \ \ \ \ \  \times \left\{ \begin{array}{ll}
 \exp\bigg\{  c_{\rho  } \Big( x^{2+\rho}  \varepsilon_n^\rho+ x^2 \delta_n^2 +(1+x) (  x^\rho \gamma_n^\rho  +  \gamma_n^\rho      +  \delta_n  )     \Big) \bigg \} & \textrm{if $\rho \in (0, 1) $}   \\
  &         \\
 \exp\bigg\{  c  \Big( x^{3}  \varepsilon_n + x^2 \delta_n^2+ (1+x)\left(x \gamma_n +  \gamma_n|\ln \gamma_n|     +  \delta_n    \right) \Big) \bigg \} & \textrm{if $\rho =1 $}
\end{array} \right.  \nonumber \\
 \nonumber \\
&&   \leq  \left\{ \begin{array}{ll}
 \exp\bigg\{  C_{\rho  } \Big( x^{2+\rho}  \varepsilon_n^\rho+ x^2 \delta_n^2 +(1+x)\big(  x^\rho \gamma_n^\rho   +  \gamma_n^\rho      +  \delta_n +\frac{  \varepsilon_n^{ \rho(2-\rho)/4 } }{ x  ^{  \rho(2+\rho)/4 }}\big)   \Big) \bigg \}   \textrm{\ \   if $\rho \in (0, 1) $} \nonumber \\
  & \\
 \exp\bigg\{  C  \Big( x^{3}  \varepsilon_n + x^2 \delta_n^2+ (1+x)\big(x \gamma_n+   \gamma_n|\ln \gamma_n|     +  \delta_n+\frac{  \varepsilon_n^{ \rho(2-\rho)/4 } }{ x  ^{  \rho(2+\rho)/4 }}\big)    \Big) \bigg \}   \textrm{\ \   if $\rho =1,$}
\end{array} \right.    \nonumber
\end{eqnarray}
which gives the desired inequalities.

 For the case of  $0\leq x  < 1,$ the proof of Theorem \ref{th2} is similar to the case of $x=1.$   Notice that $(-S_k, \mathcal{F}_k)_{k=0,...,n}$  also satisfies   conditions (A1), (A2), and (A4).  Thus, the same inequalities hold  when   $\frac{\mathbf{P}(W_n  \geq x)}{1-\Phi \left( x\right)}$ is replaced by $\frac{\mathbf{P}(W_n \leq -x)}{ \Phi \left( -x\right)}$ for all $0\leq x =  o(\gamma_n^{-1})$. This completes the proof of
 Theorem \ref{th2}.    \hfill\qed

\subsection{Proof of Theorem \ref{th3.1}} \label{sect-proof-of-weaker-expansion}
Using Proposition \ref{lem3.5}, by an argument similar to the proof of Theorem \ref{rezultCRMD}, we obtain
the following result.
If $\rho \in (0, 1)$, then for all $0\leq x =o(\max\{ \varepsilon_n^{-1},  \kappa_n^{-1}  \})$,
\begin{eqnarray*}
&&\frac{\mathbf{P}(W_n \geq x)}{1-\Phi \left( x\right)}  \\
 &&=   \exp\bigg\{ \theta c_{\rho} \bigg( x^{2+\rho}  \varepsilon_n^\rho+ x^2 \delta_n^2   +(1+x)\Big(x^{\rho/2} \varepsilon_n^{\rho/2} + \varepsilon_n ^{\rho/(3+\rho)}   + \delta_n+    \frac{ \varepsilon_n^{ \rho(2-\rho)/4 } }{ 1+ x  ^{  \rho(2+\rho)/4 }}    \Big) \bigg) \bigg \} .
\end{eqnarray*}
Notice that the following three inequalities hold:
\begin{eqnarray*}
  x^{1+\rho/2} \varepsilon_n^{\rho/2}&\leq& x^{2+\rho}  \varepsilon_n^\rho   , \ \ \ \ \ \, x \geq  \varepsilon_n^{-  \rho/(2+\rho)} , \\
  x^{ \rho/2} \varepsilon_n^{\rho/2} &\leq& \varepsilon_n ^{\rho/(3+\rho)}   , \ \ \ \ \ 0\leq x  \leq \varepsilon_n^{-  \rho/(2+\rho)},\\
  \varepsilon_n^{\rho(2-\rho)/4} &\leq& \varepsilon_n ^{\rho/(3+\rho)}, \ \ \ \ \  \rho \in (0, 1].
\end{eqnarray*}
Therefore, for $\rho \in (0, 1)$ and all $0\leq x =o(\max\{ \varepsilon_n^{-1},  \kappa_n^{-1}  \})$,
\begin{eqnarray*}
\frac{\mathbf{P}(W_n \geq x)}{1-\Phi \left( x\right)}  =  \exp\bigg\{ \theta c_{\rho} \Big( x^{2+\rho}  \varepsilon_n^\rho+ x^2 \delta_n^2 +(1+x)\big( \varepsilon_n ^{\rho/(3+\rho)}   + \delta_n  \big) \Big) \bigg \} ,
\end{eqnarray*}
which gives the desired equality for $\rho \in (0, 1)$.

Assume that condition (A2) holds for  $\rho \geq 1.$
When $\rho \in [1, 2],$ by Markov's inequality and (\ref{ineds65}), we have for all $x\geq1,$
\begin{eqnarray}
 \mathbf{E}_x\Big[  \mathbf{1}_{\big\{(x \varepsilon_n)^{1/2} < | [M]_n - \langle M\rangle_n   |      \big\}}\Big]
 &\leq&  \frac{ 1}{ (x \varepsilon_n)^{(2+\rho)/4} }  \mathbf{E}_x\big[ | [M]_n - \langle M(x) \rangle_n |^{(2+\rho)/2}    \big] \nonumber \\
&\leq& \frac{ 1}{  x  ^{(2+\rho)/4} }  \varepsilon_n^{(3\rho -2)/4} \nonumber \\
&\leq&   \varepsilon_n^{(3\rho -2)/4}. \label{fsgjk}
\end{eqnarray}
When $\rho > 2,$    Lemma \ref{lema24}  implies that   condition (A2) also holds for $\rho=2$,   with the term $\varepsilon_n$  in condition (A2) replaced by $ 2\varepsilon_n.$  Then   (\ref{fsgjk}) with $\rho=2$ shows that
$$ \mathbf{E}_x\Big[  \mathbf{1}_{\big\{(x \varepsilon_n)^{1/2} < | [M]_n - \langle M\rangle_n   |      \big\}}\Big] \leq    2\varepsilon_n .$$
Thus, for all $\rho \geq 1,$ it holds that
\begin{eqnarray*}
 \mathbf{E}_x\Big[  \mathbf{1}_{\big\{(x \varepsilon_n)^{1/2} < | [M]_n - \langle M\rangle_n   |      \big\}}\Big] \leq \max\Big\{ \varepsilon_n^{(3\rho -2)/4} , \ 2\varepsilon_n \Big\}
  \leq 2\varepsilon_n^{ \rho/(3+\rho)}.
\end{eqnarray*}
Notice that  Lemma \ref{lema24}  also implies that condition (A2) holds  for $\rho=1.$
Therefore, by (\ref{gjmsg02}), (\ref{ineq65}) can be improved to
\begin{eqnarray}
I_2(x) &\leq& e^{\frac14} \exp\bigg\{  -\frac12 x^2+  c_3  \big( x^{2+\rho}  \varepsilon_n^\rho+ x^2 \delta_n^2   \big) \bigg \}   \mathbf{E}_x\Big[  \mathbf{1}_{\big\{(x \varepsilon_n)^{1/2} < | [M]_n - \langle M\rangle_n   |      \big\}}\Big]  \nonumber  \\
 &\leq& C\, \varepsilon_n^{ \rho/(3+\rho)} \exp\bigg\{  -\frac12 x^2+  c_3  \big( x^{3}  \varepsilon_n + x^2 \delta_n^2   \big) \bigg \}. \nonumber
\end{eqnarray}
Notice also that for   $\rho \geq 1$,
\begin{eqnarray}
\mathbf{P}\Big((M_n, \sqrt{[M]_n} ) \in \mathcal{E}_3 \Big)&\leq&  \min\Big\{c\, \varepsilon_n^\rho, \,  \frac23 x^{-2/3} e^{- 3 x^2 /4  }  \Big\}  \nonumber \\
\ \ \ &\leq&  C\, \varepsilon_n^{ \rho/(3+\rho)}\exp\Big\{-\frac12 x^2  \Big\} . \nonumber
\end{eqnarray}
By an argument similar to the proof for case $\rho \in (0, 1)$ but with the term $(x \varepsilon_n)^{\rho/2}$ in (\ref{ineq62}) replaced by  $(x \varepsilon_n)^{1/2}$,  we have   for all $0\leq x =o(\max\{ \varepsilon_n^{-1},  \kappa_n^{-1}  \})$,
\begin{eqnarray*}
\frac{\mathbf{P}(W_n \geq x)}{1-\Phi \left( x\right)} & =&  \exp\bigg\{ \theta c_{1} \Big( x^3  \varepsilon_n + x^2 \delta_n^2 +(1+x)\big(x^{\rho/2} \varepsilon_n^{\rho/2} + \varepsilon_n ^{\rho/(3+\rho)}   + \delta_n   \big) \Big) \bigg \} \\
&=& \exp\bigg\{ \theta c_{2} \Big( x^3  \varepsilon_n + x^2 \delta_n^2 +(1+x)\big(  \varepsilon_n ^{\rho/(3+\rho)}   + \delta_n   \big) \Big) \bigg \},
\end{eqnarray*}
which gives the   desired equality for $\rho \geq 1$.

\subsection{Proof of Corollary \ref{corollary02}}\label{sec7}
\setcounter{equation}{0}
To prove Corollary \ref{corollary02}, we need the following two sides bound  on the tail probabilities of the standard normal random variable:
\begin{eqnarray}\label{fgsgj1}
\frac{1}{\sqrt{2 \pi}(1+x)} e^{-x^2/2} \leq 1-\Phi ( x ) \leq \frac{1}{\sqrt{ \pi}(1+x)} e^{-x^2/2}, \ \ \ \   x\geq 0.
\end{eqnarray}
First, we prove that
\begin{eqnarray}\label{dfgkmsf}
 \limsup_{n\rightarrow \infty}\frac{1}{a_n^2}\ln \mathbf{P}\bigg( \frac{  W_n}{a_n }  \in B  \bigg) \leq  - \inf_{x \in \overline{B}}\frac{x^2}{2}.
\end{eqnarray}
For any given Borel set $B\subset \mathbf{R},$ let $x_0=\inf_{x\in B} |x|.$ Then, it is obvious that $x_0\geq\inf_{x\in \overline{B}} |x|.$
Therefore, by Theorem \ref{th3.1},
\begin{eqnarray*}
 \mathbf{P}\bigg(\frac{  W_n}{a_n }  \in B \bigg)
 &\leq&  \mathbf{P}\Big(\, \big| W_n  \big|  \geq a_n x_0\Big)\\
 &\leq&  2\Big( 1-\Phi \left( a_nx_0\right)\Big) \\
  && \   \times\exp\bigg\{c_{\rho} \bigg(  \left( a_nx_0\right)^{2+\rho} \varepsilon_n^\rho   + \left( a_nx_0\right)^2 \delta_n^2 + \left( a_nx_0\right)( \varepsilon_n^{\rho /(3+\rho)} +\delta_n )  \bigg) \bigg\}.
\end{eqnarray*}
Using   (\ref{fgsgj1}),
we deduce that
\begin{eqnarray*}
\limsup_{n\rightarrow \infty}\frac{1}{a_n^2}\ln \mathbf{P}\bigg(\frac{  W_n}{a_n }  \in B  \bigg)
 \ \leq \  -\frac{x_0^2}{2} \ \leq \  - \inf_{x \in \overline{B}}\frac{x^2}{2} ,
\end{eqnarray*}
which gives (\ref{dfgkmsf}).

Next, we prove that
\begin{eqnarray}\label{dfgk02}
\liminf_{n\rightarrow \infty}\frac{1}{a_n^2}\ln \mathbf{P}\bigg(\frac{  W_n}{a_n }  \in B  \bigg) \geq   - \inf_{x \in B^o}\frac{x^2}{2} .
\end{eqnarray}
We may assume that $B^o \neq \emptyset.$
For any $\varepsilon_1>0,$ there exists an $x_0 \in B^o,$ such that
\begin{eqnarray}
 0< \frac{x_0^2}{2} \leq   \inf_{x \in B^o}\frac{x^2}{2} +\varepsilon_1.
\end{eqnarray}
For $x_0 \in B^o,$ there exists small $\varepsilon_2 \in (0, x_0),$ such that $(x_0-\varepsilon_2, x_0+\varepsilon_2]  \subset B.$
Then it is obvious that $x_0\geq\inf_{x\in \overline{B}} x.$ Without loss of generality, we may assume that $x_0>0.$
By Theorem \ref{th3.1}, we deduce that
\begin{eqnarray*}
\mathbf{P}\bigg(\frac{  W_n}{a_n }  \in B  \bigg)   &\geq&   \mathbf{P}\Big(  W_n  \in (a_n ( x_0-\varepsilon_2), a_n( x_0+\varepsilon_2)] \Big)\\
&\geq&   \mathbf{P}\Big(  W_n   > a_n ( x_0-\varepsilon_2)   \Big)-\mathbf{P}\Big( W_n  >   a_n( x_0+\varepsilon_2) \Big).
\end{eqnarray*}
Using Theorem \ref{th3.1}  and (\ref{fgsgj1}), it follows that
\begin{eqnarray*}
\liminf_{n\rightarrow \infty}\frac{1}{a_n^2}\ln \mathbf{P}\bigg(\frac{  W_n}{a_n }  \in B  \bigg)  \geq  -  \frac{1}{2}( x_0-\varepsilon_2)^2 . \label{ffhms}
\end{eqnarray*}
Letting $\varepsilon_2\rightarrow 0,$  we get
\begin{eqnarray*}
\liminf_{n\rightarrow \infty}\frac{1}{a_n^2}\ln \mathbf{P}\bigg(\frac{  W_n}{a_n }    \in B \bigg) \ \geq\ -  \frac{x_0^2}{2}  \  \geq \   -\inf_{x \in B^o}\frac{x^2}{2} -\varepsilon_1.
\end{eqnarray*}
Because $\varepsilon_1$ can be arbitrarily small, we obtain (\ref{dfgk02}). This completes the proof of Corollary \ref{corollary02}. \hfill\qed

\section*{Acknowledgements}
Fan and Liu have been partially  supported by the
 National Natural Science Foundation of China (Grant nos.\ 11601375,  11626250,  11571052,   11401590 and  11731012), and by  Hunan Natural Science Foundation (China, grant no. 2017JJ2271). Grama and Liu have  benefitted  from the support of the French government ``Investissements d$'$Avenir" program ANR-11-LABX-0020-01.
Shao has been partially supported by Hong Kong RGC GRF 14302515.

\begin{supplement}
\stitle{Supplement to ``Self-normalized Cram\'{e}r type moderate deviations for martingales"}
\slink[doi]{COMPLETED BY THE TYPESETTER}
\sdatatype{.pdf}
\sdescription{The supplement gives the detailed proofs of Propositions 3.1 and 3.2.}
\end{supplement}

\end{document}


\begin{frontmatter}
\title{Supplement to ``Self-normalized Cram\'{e}r type moderate deviations for martingales"}
\runtitle{Self-normalized Cram\'{e}r type moderate deviations}

\begin{aug}
\author{\fnms{Xiequan} \snm{Fan}\thanksref{a,e1}\ead[label=e1,mark]{fanxiequan@hotmail.com}}
\author{\fnms{Ion} \snm{Grama}\thanksref{b,e2}\ead[label=e2,mark]{ion.grama@univ-ubs.fr}}
\author{\fnms{Quansheng} \snm{Liu}\thanksref{b,e3}\ead[label=e3,mark]{quansheng.liu@univ-ubs.fr}}
\and
\author{\fnms{Qi-Man} \snm{Shao}\thanksref{c,e4}\ead[label=e4,mark]{qmshao@sta.cuhk.edu.hk}}

\address[a]{Center for Applied Mathematics,
Tianjin University, Tianjin 300072,  China.
\printead{e1}}

\address[b]{Universit\'{e} de Bretagne-Sud, LMBA, UMR CNRS 6205,
 Campus de Tohannic, 56017 Vannes, France.
\printead{e2,e3}}

\address[c]{Department of Statistics, The Chinese University of Hong Kong, Shatin, NT, Hong Kong.
\printead{e4}}

\runauthor{X. Fan et al.}

\affiliation{Tianjin University, Universit\'{e} de Bretagne-Sud and The Chinese University of Hong Kong}

\end{aug}

\begin{abstract}
We give detailed proofs for Propositions 3.1 and 3.2 in the article ``Self-normalized Cram\'{e}r type moderate deviations for martingales".
\end{abstract}

\end{frontmatter}


\section{Proof of Proposition 3.1}
  Recall that we have the notation
$$ \zeta_i(\lambda)=\lambda \xi_i - \lambda^2 \xi_i^2/2, \,\, \eta_i(\lambda) =\zeta_i(\lambda) - \mathbf{E}_{\lambda}[\zeta_i(\lambda) |\mathcal{F}_{i-1}], \;\;  Y_{n}(\lambda)=\sum_{i=1}^n \eta_i(\lambda),$$
 and $Y(\lambda )=(Y_k(\lambda ),\mathcal{F}_k)_{k=0,...,n}.$ For simplicity, we   write  $\zeta_i,$ $\eta_i,$ $Y_{n}$, $ Y $   for $\zeta_i (\lambda),$ $\eta_i (\lambda),$ $Y_{n}(\lambda)$,  $Y(\lambda)$,  respectively. In the sequel, $\vartheta$ (different from  $\theta$) stands for real numbers satisfying $0\leq \vartheta \leq 1$ and $\varphi$ stands for the density function of the standard normal distribution.

 Let $\Delta \left\langle Y \right\rangle_{k}=\mathbf{E}_{\lambda}[\eta_i^2 | \mathcal{F}_{k-1}]$ and $\left\langle Y \right\rangle_{k}= \sum_{i\leq k}\Delta \left\langle Y \right\rangle_{i}.$
Notice that for $\rho \geq 0,$ $$|\eta_{i}|^{2+ \rho} \leq 2^{1+ \rho}(|\zeta_{i}|^{2+ \rho}+ \mathbf{E}_{\lambda}[|\zeta_{i}||\mathcal{F}_{i-1}]^{2+ \rho}). $$ Using   (A4), Lemmas 3.2 and 3.5 of  \citep{FGLS17},   we have for   all $0\leq \lambda =o(\gamma_n^{-1 }),$
\begin{eqnarray}
\mathbf{E}_{\lambda}[|\eta_{i}|^{2+ \rho} | \mathcal{F}_{i-1} ] &\leq &2^{1+ \rho} \mathbf{E}_{\lambda}[|\zeta_{i}|^{2+ \rho} + \mathbf{E}_{\lambda}[|\zeta_{i}| | \mathcal{F}_{i-1}]^{2+ \rho} | \mathcal{F}_{i-1} ] \nonumber \\
&\leq & 2^{2+ \rho} \mathbf{E}_{\lambda}[|\zeta_{i}|^{2+ \rho}  | \mathcal{F}_{i-1} ] \nonumber\\
&= & 2^{2+ \rho} \frac{\mathbf{E} [| \lambda \xi_i - \lambda^2 \xi_i^2/2|^{2+ \rho}  \exp\{ \lambda\xi_{i} -  \lambda^2 \xi_i^2/2\} | \mathcal{F}_{i-1}]}{\mathbf{E} [ \exp\{ \lambda\xi_{i} -  \lambda^2 \xi_i^2/2\} | \mathcal{F}_{i-1}] } \nonumber\\
&\leq & c_{\rho}' \frac{\mathbf{E} [| \lambda \xi_i |^{2+ \rho}   | \mathcal{F}_{i-1}] }{ 1+  O(1) \lambda^{2+\rho}  \gamma_n^\rho \, \mathbf{E}[  \xi_i ^{2} |\mathcal{F}_{i-1}] } \nonumber  \\
&\leq&   c_{\rho}' \frac{  \lambda^{2+ \rho} \gamma_n^\rho   \mathbf{E} [  \xi_i  ^{2 }   | \mathcal{F}_{i-1}] }{ 1+  O(1) \lambda^{2+\rho}  \gamma_n^\rho \, \mathbf{E}[  \xi_i ^{2} |\mathcal{F}_{i-1}] }  \nonumber  \\
&\leq&   c_{\rho}  \lambda^{2+ \rho} \gamma_n^\rho  \mathbf{E}[  \xi_i ^{2} |\mathcal{F}_{i-1}]. \label{sec876}
\end{eqnarray}
Using  Lemmas 3.2 and 3.5 of  \citep{FGLS17} again,   we obtain for all $0 \leq \lambda =o(\gamma_n^{-1}),$
\begin{eqnarray}\label{f56}
&&\left|\Delta \left\langle Y\right\rangle _k-\lambda^2 \mathbf{E}[\xi _k^2| \mathcal{F}_{k-1}]\right| \nonumber \\
 &&\leq \left| \frac{\mathbf{E}[\zeta_k^2e^{\zeta_k }|
\mathcal{F}_{k-1}]}{\mathbf{E}[e^{\zeta_k }|\mathcal{F}_{k-1}]}- \lambda^2 \mathbf{E}[\xi _k^2| \mathcal{F}_{k-1}]\right| +\left| \frac{\mathbf{E}[\zeta_{k}  e^{\zeta_{k}}|\mathcal{F}_{k-1}]^2}{\mathbf{E}[e^{\zeta_k}|\mathcal{F}_{k-1}]^2}   \right|   \nonumber \\
&&= \frac{\left|\mathbf{E}[\zeta_k^2e^{\zeta_k}|\mathcal{F}_{k-1}]- \lambda^2\mathbf{E}[\xi _k^2|\mathcal{F}_{k-1}]\mathbf{E}[e^{\zeta_k}|\mathcal{F}_{k-1}]\right| }{ |\mathbf{E}[e^{\zeta_k}|\mathcal{F}_{k-1}]|}  +\left| \frac{\mathbf{E}[\zeta_{k}  e^{\zeta_{k}}|\mathcal{F}_{k-1}]^2}{\mathbf{E}[e^{\zeta_k}|\mathcal{F}_{k-1}]^2}   \right| \nonumber\\
&&\leq c_1  \Big(  \mathbf{E}[|\lambda \xi _k|^{2+\rho} | \mathcal{F}_{k-1}]  +   (\lambda^2 \mathbf{E}[\xi _k^2| \mathcal{F}_{k-1}] )^2  \Big)     \nonumber\\
&&\leq c_2\, \lambda^{2+\rho}\gamma_n^\rho  \mathbf{E}[\xi _k^2| \mathcal{F}_{k-1}]. \label{sec8fs7}
\end{eqnarray}
Therefore,
\begin{eqnarray}
|\langle Y\rangle_{n} -  \lambda^2  | &\leq& |\langle Y\rangle_{n} - \lambda^2\langle M\rangle_{n}| + \lambda^2|\langle M\rangle_{n}-1| \nonumber \\
 &\leq& c\,  \lambda^{2+\rho}\gamma_n^\rho   \langle M\rangle_{n}  + \lambda^2\delta_n^2 . \label{sec877}
\end{eqnarray}
Inequalities (\ref{sec876}) and (\ref{sec877}) show that the martingale $Y$
satisfies the following conditions.   For all $0< \lambda =o(\gamma_n^{-1 }),$
\begin{description}
\item[(B1)] $\mathbf{E}_{\lambda}\big[|\eta_{i}/\lambda|^{2+ \rho} \,  \big | \mathcal{F}_{i-1} \big] \leq  c_\rho \,  \gamma_n^\rho \,  \mathbf{E}[\xi_{i}^2 | \mathcal{F}_{i-1} ]$;
\item[(B2)] $\; \left| \left\langle Y\right\rangle _n/\lambda^2- 1\right| \leq  c\,    \lambda^{ \rho}\gamma_n^\rho  +  \delta_n^2.$
\end{description}

For simplicity of notation, set $T=1+\delta_n ^2.$ We introduce a modification of the conditional variance of the
martingale $M$ as follows:
\begin{equation}
V_k=\left\langle M\right\rangle _k\mathbf{1}_{\{k<n\}}+T\mathbf{1}_{\{k=n\}}.
\label{RA-4}
\end{equation}
It is easy to see that $V_0=0$ and $V_n=T,$ and that $(V_k,\mathcal{F}_k)_{k=0,...,n}$ is a
predictable process. Denote
\begin{displaymath}
\widetilde{\gamma}_n  = \lambda^\rho \gamma_n^\rho   + \gamma_n  +\delta_n.
\end{displaymath}
  Let $
c_{*}$ be a constant depending only on $\rho$, whose exact value will be
chosen later. Then $$A_k=c_{*}^2\widetilde{\gamma}_n ^2+T-V_k, \  \ k=1,...,n,$$
is a non-increasing  predictable process.
For any fixed $u, x\in \mathbf{R},$ and $
y > 0,$ set, for brevity,
\begin{equation}
\Phi _u(x,y)=\Phi \Big( \frac{u-x}{\sqrt{y} }  \Big) .  \label{RA-7}
\end{equation}

In the proof we make use of the following two  lemmas from   \citep{Bo82}.
\begin{lemma}
\label{LEMMA-APX-1}Let $X$ and $Y$ be random variables. Then
\[
\sup_u\Big| \mathbf{P}\left( X\leq u\right) -\Phi \left( u \right) \Big|
\leq c_1\sup_u\Big| \mathbf{P}\left( X+Y\leq u\right) -\Phi \left( u \right)
\Big| +c_2 \Big\| \mathbf{E}\left[ Y^2|X\right] \Big\| _\infty ^{1/2}.
\]
\end{lemma}
\begin{lemma}
\label{LEMMA-APX-2}Let $G(x)$ be an integrable function on $\mathbf{R}$ of bounded variation $||G||_V$,
$X$ be a random variable, and $a,$ $b\neq 0$ be real numbers. Then
\[
\mathbf{E}\bigg[ \, G\left( \frac{X+a}b\right) \bigg] \leq ||G||_V \sup_u\Big| \mathbf{P}\left( X\leq u\right)
-\Phi \left( u\right) \Big| + ||G||_1 \, \Big|b \Big|,
\]
where  $||G||_1$ is the $L_1(\mathbf{R})$ norm of $G(x).$
\end{lemma}

Let $\mathcal{N}  =\mathcal{N}(0, 1)$ be a standard normal random variable  independent of $Y_n$.
Using Lemma \ref{LEMMA-APX-1},
we deduce that
\begin{eqnarray}
 \sup_u\Big| \mathbf{P}_{\lambda}(Y_n/\lambda \leq u)  -  \Phi (u)\Big| &\leq& c_0\sup_u\Big| \mathbf{P}_{\lambda}(c_* \widetilde{\gamma}_n\mathcal{N}  + Y_n/\lambda \leq u)  -  \Phi (u)\Big|+ c_1 c_* \widetilde{\gamma}_n \nonumber\\
&=& c_1\sup_u\Big| \mathbf{E}_{\lambda}[\Phi _u(Y_n/\lambda ,A_n)]-\mathbf{E}_{\lambda}[\Phi _u(Y_0/\lambda ,A_0)]\Big|+c_2\widetilde{\gamma}_n \nonumber\\
&\leq& c_1\sup_u\Big| \mathbf{E}_{\lambda}[\Phi _u(Y_n/\lambda ,A_n)]-\mathbf{E}_{\lambda}[\Phi _u(Y_0/\lambda ,A_0)]\Big| \nonumber\\
& & +\, c_1\sup_u\Big| \mathbf{E}_{\lambda}[\Phi _u(Y_0/\lambda ,A_0)]-\Phi (u)\Big| +c_2\widetilde{\gamma}_n   \nonumber\\
&=& c_1\sup_u\Big| \mathbf{E}_{\lambda}[\Phi _u(Y_n/\lambda ,A_n)]-\mathbf{E}_{\lambda}[\Phi _u(Y_0/\lambda ,A_0)]\Big| \nonumber\\
& & +\, c_1\sup_u\Big| \Phi \Big(\frac{u}{\sqrt{c_{*}^2\widetilde{\gamma}_n^2+T}}  \Big)-\Phi (u)\Big| +c_2\widetilde{\gamma}_n , \label{llaadf}
\end{eqnarray}
where the last line follows from the fact that $Y_0=0$ and $A_0=c_{*}^2 \widetilde{\gamma}_n ^2+T.$ Because $T=1+\delta_n ^2,$ it is obvious that
\[
\sup_u\Big| \Phi \Big(\frac{u}{\sqrt{c_{*}^2\widetilde{\gamma}_n^2+T}}  \Big)-\Phi (u)\Big| \leq c_3  \,  \widetilde{\gamma}_n.
\]
Thus, from (\ref{llaadf}),
\begin{eqnarray}
 \sup_u\Big| \mathbf{P}_{\lambda}(Y_n/\lambda \leq u)  -  \Phi (u)\Big|
 \leq  c_1\sup_u\Big| \mathbf{E}_{\lambda}[\Phi _u(Y_n/\lambda ,A_n)]-\mathbf{E}_{\lambda}[\Phi _u(Y_0/\lambda ,A_0)]\Big|  +c_4 \widetilde{\gamma}_n. \label{llaa}
\end{eqnarray}
For the first item on the right hand side of the last inequality,   we have the following telescoping
\[
\mathbf{E}_\lambda [\Phi _u(Y_n/\lambda,A_n)]-\mathbf{E}_\lambda [\Phi _u(Y_0/\lambda,A_0)]=\mathbf{E}_\lambda
\Big[ \sum_{k=1}^n \Big(\Phi _u(Y_k/\lambda,A_k)-\Phi _u(Y_{k-1}/\lambda,A_{k-1}) \Big)\Big] .
\]
Taking into account that $(\eta _i,\mathcal{F}_i)_{i=0,...,n}$ is
a $\mathbf{P}_\lambda $-martingale and that
\[
\frac{\partial ^2}{\partial x^2}\Phi _u(x,y)=2\frac \partial {\partial
y}\Phi _u(x,y),
\]
we deduce that
\begin{equation}
\mathbf{E}_\lambda[ \Phi _u(Y_n/\lambda,A_n)]-\mathbf{E}_\lambda [\Phi _u(Y_0/\lambda,A_0)]=I_1+I_2-I_3,
\label{lla}
\end{equation}
where
\begin{eqnarray}
I_1&=&\mathbf{E}_\lambda \bigg[ \sum_{k=1}^n \bigg(\frac{}{} \Phi _u(Y_k/\lambda,A_k)-\Phi
_u(Y_{k-1}/\lambda,A_k)  \nonumber \\
&& \ \ \ \ \ \    -\frac \partial {\partial x}\Phi
_u(Y_{k-1}/\lambda,A_k)\frac{\eta_k}{\lambda }-\frac 12\frac{\partial ^2}{\partial x^2}\Phi
_u(Y_{k-1}/\lambda,A_k) \frac{\eta_k^2}{\lambda^2 } \bigg) \bigg] ,  \nonumber \\
I_2 &=& \frac 12\mathbf{E}_\lambda  \bigg[  \sum_{k=1}^n\frac{\partial ^2}{\partial x^2}\Phi
_u(Y_{k-1}/\lambda,A_k)\Big(\Delta \left\langle Y\right\rangle_k/\lambda^2-\Delta V_k \Big) \bigg] ,\quad \quad \
\nonumber  \\
I_3 &=&\mathbf{E}_\lambda  \bigg[\sum_{k=1}^n\Big( \Phi _u(Y_{k-1}/\lambda,A_{k-1})-\Phi
_u(Y_{k-1}/\lambda,A_k)-\frac \partial {\partial y}\Phi _u(Y_{k-1}/\lambda,A_k)\Delta
V_k\Big)  \bigg].  \nonumber
\end{eqnarray}
Next, we   estimate  $I_1,$ $I_2,$ and $I_3.$ To shorten the notations, denote
\[T_{k-1}=\frac{1}{\sqrt{A_k}} \Big(u-\frac{Y_{k-1}}{\lambda}  \Big).\]

\emph{\textbf{a)} Control of }$I_1 .$
 Assume that $f$  is a three times differentiable function on $\mathbf{R}$.
By Taylor's expansion,   it is easy to see that  for any $|\Delta x|\leq 1,$
\begin{eqnarray*}
\Big|f(x+\Delta x)-f(x) - f'(x)  \Delta x - \frac12 f''(x) (\Delta x)^2 \Big|  & =&   \Big|\frac16 f'''(x+ \vartheta \Delta x ) (\Delta x)^3 \Big| \nonumber \\
 &\leq&  |f'''(x+ \vartheta \Delta x )| \, |\Delta x|^{2+\rho} ,
\end{eqnarray*}
and for any $|\Delta x|>1,$
\begin{eqnarray*}
&&\Big|f(x+\Delta x)-f(x) - f'(x)  \Delta x - \frac12 f''(x) (\Delta x)^2 \Big|   \nonumber \\
&& =   \Big|\frac12 f''(x+ \vartheta_1 \Delta x ) (\Delta x)^2 - \frac12 f''(x) (\Delta x)^2 \Big| \nonumber \\
 &&\leq \frac12 \Big(|f''(x+ \vartheta_1 \Delta x )| + |f''(x)|  \Big)  \, |\Delta x|^{2}  \nonumber \\
  &&\leq |f''(x+ \vartheta \Delta x )|   \, |\Delta x|^{2} \nonumber \\
   &&\leq |f''(x+ \vartheta \Delta x )|   \, |\Delta x|^{2+\rho}  .
\end{eqnarray*}
Taking $f(x)=\Phi(x),$ $x=T_{k-1}$ and $\Delta x= \frac{\eta_{k}}{ \lambda \sqrt{A_{k}}},$
  we have
\begin{eqnarray}
|I_1|&\leq& \mathbf{E}_\lambda \bigg[ \sum_{k=1}^n  F\Big(T_{k-1}+ \frac{\vartheta_1\eta_{k}}{\lambda \sqrt{ A_{k}}} \Big)  \Big|\frac{\eta_{k}}{\lambda \sqrt{A_{k}}} \Big|^{2+\rho} \mathbf{1}_{\left\{|\eta_{k}/\lambda\sqrt{A_{k}}| \leq  1+ |T_{k-1}| /2  \right\}}\bigg] \nonumber \\
 && \   + \, \mathbf{E}_\lambda \bigg[ \sum_{k=1}^n \Big|\frac{}{} \Phi _u(Y_k/\lambda,A_k)-\Phi
_u(Y_{k-1}/\lambda,A_k)   -\frac \partial {\partial x}\Phi
_u(Y_{k-1}/\lambda,A_k)\frac{\eta_k}{\lambda }\nonumber \\
&&\   \   \   \   \   \   \    \   \    \   \   \   \  \ \ \ \ \  \ \ \ \   \     \   -\frac 12\frac{\partial ^2}{\partial x^2}\Phi
_u(Y_{k-1}/\lambda,A_k) \frac{\eta_k^2}{\lambda^2 }\Big|   \mathbf{1}_{\left\{|\eta_{k}/\lambda \sqrt{A_{k}}| >  1+ |T_{k-1}| /2  \right\}}\bigg] \nonumber \\
& :=& I_{11} + I_{12} ,  \label{CC-1}
\end{eqnarray}
where
$$F\big(t \big)= \max\Big\{ \Big|\Phi'''\big(t \big)\Big| , \Big|\Phi''\big(t \big) \Big|\Big\}.$$
To bound the right hand side of (\ref{CC-1}), we distinguish two cases as follows.

\emph{Case 1}: $|\eta_{k}/\lambda \sqrt{A_{k}}| \leq 1+   |T_{k-1}| /2$.
By the inequality $ F(t)  \leq\varphi( t)(1+t^2)$, it follows that
\begin{eqnarray*}
   F\Big(T_{k-1}+ \frac{\vartheta_1\eta_{k}}{\lambda\sqrt{A_{k}}} \Big )  &\leq & \varphi\Big(T_{k-1}+\frac{\vartheta_1\eta_{k}}{\lambda \sqrt{A_{k}}} \Big)\bigg(1+\Big(T_{k-1}+\frac{\vartheta_1\eta_{k}}{ \lambda\sqrt{A_{k}}} \Big)^2 \bigg) \\
  &\leq &  g_{1}(T_{k-1}) ,
\end{eqnarray*}
where $$g_{1}(z)=\sup_{|t-z|  \leq   1+ |z| /2 }\varphi(t )(1+ t^2).$$
It is easy to see that $g_{1}(z)$ is    non-increasing in $z\geq 0$. Because $g_{1}(z)$ is nonnegative,
\begin{eqnarray}\label{fklm59}
I_{11} \leq \mathbf{E}_\lambda \bigg[\sum_{k=1}^n g_{1}(T_{k-1}) \Big|\frac{\eta_{k}}{\lambda \sqrt{A_{k}}} \Big|^{2+\rho} \mathbf{1}_{\left\{|\eta_{k}/\lambda\sqrt{A_{k}}| \leq  1+ |T_{k-1}| /2  \right\}}\bigg] .
\end{eqnarray}

\emph{Case 2}: $|\eta_{k}/\lambda\sqrt{A_{k}}| > 1+ |T_{k-1}| /2$.
It is easy to see that
\begin{eqnarray*}
&&\Big|\Phi(x+\Delta x)-\Phi(x) - \Phi'(x)  \Delta x - \frac12 \Phi''(x) (\Delta x)^2 \Big|   \nonumber \\
& &=  \Big(  \Big| \frac{\Phi(x+\Delta x)-\Phi(x)}{|\Delta x|^{2+\rho}}  \Big|       +  | \Phi'(x)|   + \Big| \frac12 \Phi''(x) \Big|  \Big) \, |\Delta x|^{2+\rho} \nonumber \\
&&\leq   \Big(4 \Big| \frac{\Phi(x+\Delta x)-\Phi(x)}{(2+ |x|)^{2}}  \Big| +  | \Phi'(x)| + |\Phi''(x)|  \Big)  \, |\Delta x|^{2+\rho} \nonumber \\
&&\leq   \Big(  \frac{c_1}{(2+ |x|)^{2}}+  | \Phi'(x)| + |\Phi''(x)|  \Big)  \, |\Delta x|^{2+\rho}\nonumber \\
&&\leq   \frac{c_2}{(2+ |x|)^{2}}   \, |\Delta x|^{2+\rho}
\end{eqnarray*}
for $|\Delta x|>1+   |x|/2.$
Because $|\Phi''(t)| \leq 2$, it follows that
\begin{eqnarray}\label{fktm59}
I_{12} \leq \mathbf{E}_\lambda \,\bigg[\sum_{k=1}^n g_{2}(T_{k-1})    \Big|\frac{\eta_{k}}{\lambda \sqrt{A_{k}}} \Big|^{2+\rho}  \mathbf{1}_{\left\{|\eta_{k}/\lambda\sqrt{A_{k}}| >  1+ |T_{k-1}| /2  \right\}}\bigg],
\end{eqnarray}
where $$g_{2}(z) =   \frac{c_2}{(2+|z|)^2} .$$

Set $$G(t)=  g_{1}(t)+g_{2}(t) .$$
It follows that
\begin{eqnarray}\label{ineq90}
|I_1|  \leq   I_{11} + I_{12} \leq \mathbf{E}_\lambda \,\bigg[\sum_{k=1}^n G(T_{k-1})    \Big|\frac{\eta_{k}}{\lambda \sqrt{A_{k}}} \Big|^{2+\rho}  \bigg]  .
\end{eqnarray}
Now we consider the conditional expectation of $|\eta_{k}|^{2+\rho}$. Using condition (B1), we have
\[
\mathbf{E}_{\lambda}[|\eta_{k}|^{2+\rho} | \mathcal{F}_{k-1}] \leq c_\rho \,\lambda^{2+\rho}\, \gamma_n^{\rho} \,\Delta \langle M \rangle_{k},
\]
where $\Delta \langle M \rangle_{k}=\langle M \rangle_{k}-\langle M \rangle_{k-1}$.
From the definition of the process $V$, it follows  that $\Delta \left\langle M\right\rangle _k=\Delta V_k=V_{k}-V_{k-1}, 1\leq k <n,$ and $
\Delta \left\langle M\right\rangle_n\leq \Delta V_n$, and that
\begin{equation} \label{ee}
\mathbf{E}_{\lambda}[|\eta_{k}|^{2+\rho}  | \mathcal{F}_{k-1}] \leq c_\rho \lambda^{2+\rho} \gamma_n^{\rho}\, \Delta V_k.
\end{equation}
 Returning to (\ref{ineq90}), by inequality (\ref{ee}),  we get
\begin{eqnarray}
\left| I_1\right| \leq   J_1,
\end{eqnarray}
 where
\begin{equation}
J_1=c_\rho\, \gamma_n^{\rho}  \mathbf{E}_\lambda  \Big[\sum_{k=1}^n\frac 1{A_k^{1+ \rho/2}}G\left( T_{k-1} \right) \Delta V_k \, \Big].  \label{W-0-0}
\end{equation}
We introduce the time change $\tau _t$ as follows. For any real $t\in [0,T]$,
\begin{equation}
\tau _t=\min \{k\leq n:V_k>t\},\quad \textrm{where}\quad \min \emptyset =n.
\end{equation}
Let $(\sigma _k)_{k=1,...,n+1}$ be the increasing sequence of
moments when the increasing stepwise function $\tau _t,$ $t\in [0,T]$, has
jumps. It is clear that $\Delta V_k=\int_{[\sigma _k,\sigma _{k+1})}dt,$ and
that $k=\tau _t$ for $t\in [\sigma _k,\sigma _{k+1}).$ Because $\tau
_T=n,$ we have
\begin{eqnarray*}
\sum_{k=1}^n\frac 1{A_k^{1+ \rho/2}}\, G\left( T_{k-1}\right)
\Delta V_k &=&\sum_{k=1}^n\int_{[\sigma _k,\sigma _{k+1})}\frac 1{A_{\tau
_t}^{1+ \rho/2}}\, G\left( T_{ \tau_t-1}  \right) dt \\
&=&\int_0^T\frac 1{A_{\tau _t}^{1+ \rho/2}}\, G\left( T_{ \tau_t-1} \right) dt.
\end{eqnarray*}
Set  $a_t=c_{*}^2\widetilde{\gamma}_n ^2+T-t.$  Because $\Delta V_{\tau _t}\leq  \gamma_n ^2+2 \delta_n^2$ (cf.,\ Lemma 3.5 of \citep{FGLS17} and (\ref{RA-4})), we see that
\begin{equation}
t\leq V_{\tau _t}\leq V_{\tau _t-1}+\Delta V_{\tau _t}\leq t+ \gamma_n
^2+2 \delta_n^2,\quad t\in [0,T].  \label{BOUND-V}
\end{equation}
Assume that   $c_{*}\geq 4.$ We have
\begin{equation}
\frac 12 \, a_t\leq c_{*}^2\widetilde{\gamma}_n ^2+T-(t+ \gamma_n
^2+2 \delta_n^2)\leq A_{\tau _t}= c_{*}^2\widetilde{\gamma}_n ^2+T-V_{\tau_t}\leq a_t,\quad
t\in [0,T].  \label{BOUND-A}
\end{equation}
Notice that $G(z)$ is symmetric, and is non-increasing in $z\geq 0.$  From (\ref{W-0-0}), (\ref{BOUND-A})
implies that
\begin{equation}
J_1\leq c_\rho \, \gamma_n^{\rho} \int_0^T\frac 1{a_t^{1+ \rho/2}} \, \mathbf{E}_\lambda \bigg[ G\bigg( \frac{
u-Y_{\tau _t-1}/\lambda}{a_t^{1/2}}\bigg) \bigg] dt.  \label{W-0}
\end{equation}
By Lemma \ref{LEMMA-APX-2}, it is easy to see that
\begin{equation}
\mathbf{E}_\lambda \bigg[ G \left( \frac{u-Y_{\tau _t-1}/\lambda}{{a_t}^{1/2}}\right)  \bigg] \leq
c_{\rho, 1}\sup_z\Big| \mathbf{P}_{\lambda}(Y_{\tau _t-1}/\lambda\leq z)-\Phi (z)\Big| +c_{\rho, 2}\sqrt{a_t}.
\label{W-1}
\end{equation}
Because $V_{\tau _t-1}=V_{\tau _t}-\Delta V_{\tau _t}$, $ V_{\tau _t}\geq t$   and  $\Delta V_{\tau _t}\leq
\gamma_n^2 + 2\,\delta_n^2$, we get
\begin{equation}
V_n-V_{\tau _t-1}\leq V_n - V_{\tau _t}+\Delta V_{\tau _t} \leq T-t+c_*(\gamma_n^2 + \delta_n^2)\leq a_t.  \label{CC-2}
\end{equation}
By (\ref{sec876}),  we have
\[ \mathbf{E}_\lambda [\eta _k^2/\lambda^2 |\mathcal{F}_{k-1}] \leq c_\rho \Delta \left\langle M\right\rangle _k.  \]
   Thus
\begin{eqnarray*}
\mathbf{E}_\lambda \left[(Y_n/\lambda-Y_{\tau _t-1}/\lambda)^2|\mathcal{F}_{\tau _t-1} \right] &=&\mathbf{E}_{\lambda}\bigg[\sum_{k=\tau_t}^n\mathbf{E}_\lambda [\eta _k^2/\lambda^2 |\mathcal{F}_{k-1}]
\bigg|\mathcal{F}_{\tau_t -1}\bigg] \\
&\leq &c_\rho\, \mathbf{E}_{\lambda} \bigg[ \sum_{k=\tau_t}^n\Delta \left\langle M\right\rangle _k
\bigg| \mathcal{F}_{\tau_t -1}\bigg] \\
&= &c_\rho\, \mathbf{E}_{\lambda}\big[ \left\langle M\right\rangle _n-\left\langle
M\right\rangle_{\tau _t-1}|\mathcal{F}_{\tau_t -1}\big]\\
&\leq& c_\rho\, \mathbf{E}_{\lambda}  [V_n-V_{\tau _t -1}|\mathcal{F}_{\tau_t -1} ]    \\
&\leq& c_\rho\, a_t.
\end{eqnarray*}
Then, by Lemma \ref{LEMMA-APX-1}, we deduce that for any $t\in [0,T],$%
\begin{equation}
\sup_z\Big| \mathbf{P}_{\lambda}(Y_{\tau _t-1}/\lambda\leq z)-\Phi (z)\Big| \leq c_{\rho, 3} \, \sup_z\Big|
\mathbf{P}_{\lambda}(Y_n/\lambda\leq z)-\Phi (z)\Big| +c_{\rho, 4}\sqrt{a_t}.  \label{W-2}
\end{equation}
From (\ref{W-0}), using (\ref{W-1}) and (\ref{W-2}), we obtain
\begin{equation}
J_1\leq c_{\rho, 5}\, \gamma_n^\rho \int_0^T\frac{dt}{a_t^{1+ \rho/2}}\sup_z\Big| \mathbf{P}_{\lambda}(Y_n/\lambda\leq
z)-\Phi (z)\Big| +c_{\rho, 6}\, \gamma_n^\rho \int_0^T\frac{dt}{a_t^{ (1+\rho)/2}}.  \label{COMP-1}
\end{equation}
By (\ref{CC-2}) and some elementary computations, we see that
\begin{equation}
\int_0^T\frac{dt}{a_t^{1+ \rho/2}} \leq \int_0^T\frac{dt}{(c_{*}^2\widetilde{\gamma}_n ^2+T-t)^{1+ \rho/2}} \leq
\frac {c_\rho}{c_{*}^\rho \gamma_n^\rho }  \, ,  \label{W-3}
\end{equation}
and
\begin{displaymath}
\int_0^T\frac{dt}{a_t^{ (1+\rho)/2}} \leq \left\{ \begin{array}{ll}
c_\rho, & \textrm{\ \ \ if $\rho \in (0, 1),$}\\
c \left| \ln
  \gamma_n  \right|, & \textrm{\ \ \ if $\rho=1.$}
\end{array} \right.
\end{displaymath}
Then
\begin{equation}
\left| I_1\right| \leq J_1\leq \frac{c_{ \rho,7} }{c_{*}^\rho}\sup_z\Big| \mathbf{P}(Y_n/\lambda\leq
z)-\Phi (z)\Big| +c_{ \rho, 8}\, \widehat{\gamma}_n\, ,  \label{W-4}
\end{equation}
where
\begin{displaymath}
\widehat{\gamma}_n  = \left\{ \begin{array}{ll}
  \lambda^\rho \gamma_n^\rho+\gamma_n^\rho  +\delta_n, & \textrm{\ \ \ if $\rho \in (0, 1),$}\\
   \lambda  \gamma_n +\gamma_n \left| \ln   \gamma_n  \right|+ \delta_n , & \textrm{\ \ \ if $\rho=1.$}
\end{array} \right.
\end{displaymath}

\emph{\textbf{b)} Control of} $I_2.$
Set $\widetilde{G}(z)=\sup_{\left| v\right| \leq 2}\psi (z+v),$ where $\psi (z)=\varphi (z)(1+z^2)^{3/2}.$
 Because $\Delta A_k=-\Delta V_k,$ we have $\left|
I_2\right| \leq I_{2,1}+I_{2,2},$ where
\begin{eqnarray*}
I_{2,1}&=&\mathbf{E}_\lambda  \bigg[\sum_{k=1}^n\frac 1{2A_k}\left| \varphi ^{\prime }\left(
T_{k-1}\right) \left( \Delta V_k-\Delta
\left\langle M\right\rangle _k\right) \right|  \bigg],\\
I_{2,2}&=&\mathbf{E}_\lambda \bigg[ \sum_{k=1}^n\frac 1{2A_k}\left| \varphi ^{\prime }\left(T_{k-1}\right) \left( \Delta \left\langle
Y\right\rangle _k /\lambda^2 -\Delta \left\langle M\right\rangle _k\right) \right| \bigg].
\end{eqnarray*}
We first deal with $I_{2,1}$. Because $\left| \varphi ^{\prime }(z)\right| \leq \widetilde{G}(z)$ for any real $z,$ we have
\begin{equation}
\left| \varphi ^{\prime }\left(T_{k-1}\right) \right|
\leq \widetilde{G}\left( T_{k-1}\right) .  \label{Bound-G}
\end{equation}
Notice that $0\leq \Delta V_k -\Delta \left\langle M\right\rangle _k\leq
2\delta_n ^2\mathbf{1}_{\{k=n\}}$, $A_n = c_{*}^2\widetilde{\gamma}_n ^2$, and $c_{*}\geq 4$. Then, using (\ref{Bound-G}), we get the estimations
\[
I_{2,1}\leq \frac{c_{2,\rho} \delta_n ^2}{c_{*}^2\widetilde{\gamma}_n ^2  } \ \mathbf{E}_\lambda [ \widetilde{G}\left( T_{n-1}
\right) ] \leq  \frac{c_{2,\rho}  }{c_{*}^2  } \ \mathbf{E}_\lambda [ \widetilde{G}\left( T_{n-1}
\right) ],
\]
and, by (\ref{W-1}) with $G=\widetilde{G}$ and (\ref{W-2}) with $t=T,$
\[
\left| I_{2,1}\right| \leq \frac{c_{3,\rho}}{c_{*}}\sup_z\Big| \mathbf{P}_{\lambda}(Y_n/\lambda\leq z)-\Phi
(z)\Big| +c_{4,\rho}\widetilde{\gamma}_n .
\]
We next consider $I_{2,2}.$ By (\ref{sec8fs7}),  we easily obtain the bound
\[
\left| \Delta \left\langle Y\right\rangle _k/\lambda^2-\Delta \left\langle
M\right\rangle _k\right| \leq c_3\,  \lambda^\rho \gamma_n^\rho \Delta \langle M\rangle_{k}  \leq
c_3\, \lambda^\rho \gamma_n^\rho   \Delta V_k  \ .
\]
With this bound, we get
\[
\left| I_{2,2}\right| \leq c_3\,  \lambda^\rho \gamma_n^\rho  \mathbf{E}_\lambda \Big[ \sum_{k=1}^n\frac
1{2A_k}\left|\varphi ^{\prime }\left(T_{k-1}\right)\right|
\Delta V_k \Big].
\]
Because $\left| \varphi ^{\prime }(z)\right| \leq \widetilde{G}(z),$ the right-hand
side can be bounded exactly in the same way as $J_1$ in (\ref{W-0-0}), with $A_k$
 replacing $A_k^{1+ \rho/2}.$ Similar to the proof of (\ref{COMP-1}), we get
\[
\left| I_{2,2}\right| \leq c_{5, \rho} \, \lambda^\rho \gamma_n^\rho \, \int_0^T\frac{dt}{a_t}%
\sup_z\Big| \mathbf{P}_{\lambda}(Y_n/\lambda\leq z)-\Phi (z)\Big| +c_{6, \rho}\, \lambda^\rho \gamma_n^\rho  \int_0^T%
\frac{dt}{a_t^{1/2}}.
\]
By some elementary computations, we see that
\[
\int_0^T\frac{dt}{a_t^{1/2}}\leq \int_0^T\frac{dt}{\sqrt{T-t}}\leq c_2,
\]
and, taking into account that $a_t\geq c_{*}^2\widetilde{\gamma}_n ^2,$
\[
\int_0^T\frac{dt}{a_t}\leq   |\ln c_{*}^2\widetilde{\gamma}_n ^2 |  \leq  c_{\rho}  |\ln \lambda \gamma_n  | .
\]
Then
\[
\left| I_{2,2}\right| \leq \frac{c_{7, \rho}}{c_{*}}  \sup_z\Big| \mathbf{P}_{\lambda}(Y_n/\lambda\leq
z)-\Phi (z)\Big| +c_{8, \rho}\,\lambda^\rho \gamma_n^\rho  .
\]
Combining the bounds   $I_{2,1}$ and $I_{2,2}$, we get
\begin{equation}
\left| I_2\right| \leq \frac{c_{9, \rho}}{c_{*}}\sup_z\Big| \mathbf{P}_{\lambda}(Y_n/\lambda\leq z)-\Phi
(z)\Big| +c_{10, \rho} \widetilde{\gamma}_n .  \label{F-2}
\end{equation}

\emph{\textbf{c)} Control of} $I_3.$  By Taylor's expansion, it follows that
\[
I_3=\frac{1}{8}\,\mathbf{E}_\lambda  \bigg[\sum_{k=1}^n\frac 1{(A_k-\vartheta _k\Delta A_k)^2}\varphi
^{\prime \prime \prime }\left( \frac{u-Y_{k-1}/\lambda}{\sqrt{A_k-\vartheta _k\Delta A_k} }
\right) \Delta A_k^2\bigg].
\]
Because $\left| \Delta A_k\right| =  \Delta V_k  \leq \gamma_n^2 + 2\,\delta_n^2$
and $c_{*}\geq 4,$ we have
\begin{equation}
A_k\leq A_k-\vartheta _k\Delta A_k\leq c_{*}^2\widetilde{\gamma}_n ^2+T-V_k+ 2\gamma_n ^2\leq
2A_k. \label{eqsfion}
\end{equation}
Using (\ref{eqsfion}) and the inequalities $\left| \varphi ^{\prime \prime \prime
}(z)\right|  \leq \widetilde{G}(z)$, we obtain
\[
\left| I_3\right| \leq c \,(\gamma_n^2 + 2\,\delta_n^2)\mathbf{E}_\lambda \bigg[ \sum_{k=1}^n\frac 1{A_k^2} \ \widetilde{G}\left(
\frac{T_{k-1}}{\sqrt{2 }}\right) \Delta V_k \bigg].
\]
Proceeding in the same way as for estimating $J_1$ in (\ref{W-0-0}), we get
\begin{equation}
\left| I_3\right| \leq \frac{c_{11,\rho}}{c_{*}}\sup_z\Big| \mathbf{P}_{\lambda}(Y_n/\lambda\leq z)-\Phi
(z)\Big| +c_{12,\rho}\widetilde{\gamma}_n .  \label{F-3}
\end{equation}

 From  (\ref{lla}),  using (\ref{W-4}), (\ref{F-2}), and (\ref{F-3}), we have
\begin{eqnarray*}
 \Big| \mathbf{E}_\lambda [\Phi _u(Y_n/\lambda,A_n)]-\mathbf{E}_\lambda [\Phi _u(Y_0/\lambda,A_0)] \Big|    \leq  \frac{c'_{1,\rho}}{c_{*}^\rho}\sup_z\Big| \mathbf{P}_{\lambda}(Y_n/\lambda\leq z)-\Phi
(z)\Big|   +\, c'_{2,\rho}  \widehat{\gamma}_n.
\end{eqnarray*}
Implementing the last bound in (\ref{llaa}), we obtain
\begin{eqnarray*}
\sup_z\Big| \mathbf{P}_{\lambda}(Y_n/\lambda\leq z)-\Phi (z)\Big|  \leq  \frac{c'_{3,\rho}}{c_{*}^\rho}\sup_z\Big|
\mathbf{P}_{\lambda}(Y_n/\lambda\leq z)-\Phi (z)\Big|   +c'_{4,\rho} \,\widehat{\gamma}_n,
\end{eqnarray*}
from which, choosing $c_{*}^\rho=\max\{ 2c'_{3,\rho}, 4^\rho\}$, we get
\begin{equation}
\sup_z\Big| \mathbf{P}_{\lambda}(Y_n/\lambda\leq z)-\Phi (z)\Big| \leq 2\, c'_{4,\rho}  \,  \widehat{\gamma}_n,
\label{F-4}
\end{equation}
which gives the  desired inequalities.

\section{Proof of Proposition 3.2}
Assume   conditions (A1), (A2), and (A3).
By Lemma  3.2 of  \citep{FGLS17}, and  condition (A3), it follows that for all $0\leq \lambda =o(  \kappa_n^{-1}  ),$
\begin{eqnarray*}
\mathbf{E} [ \exp\{ \lambda\xi_{i} -  \lambda^2 \xi_i^2/2\} | \mathcal{F}_{i-1}] &=& 1+ O(1) \lambda^2 \mathbf{E} [ \xi_{i}^2  | \mathcal{F}_{i-1}] \\
&=& 1 +o(1).
\end{eqnarray*}
By an argument similar to   the proof of (\ref{sec876}), we get for all $0\leq \lambda =o(\max\{ \varepsilon_n^{-1},  \kappa_n^{-1}  \}),$
\begin{eqnarray}
\sum_{i=1}^n\mathbf{E}_{\lambda}[|\eta_{i}|^{2+ \rho} | \mathcal{F}_{i-1} ]
&\leq&   c_{\rho} \sum_{i=1}^n   \mathbf{E} [| \lambda \xi_i |^{2+ \rho}   | \mathcal{F}_{i-1}] \nonumber  \\
&\leq& c_{\rho}  \lambda ^{2+ \rho} \varepsilon_n^\rho \,.
\end{eqnarray}
Similarly, we have for all $0\leq \lambda =o(\max\{ \varepsilon_n^{-1},  \kappa_n^{-1}  \}),$
\begin{eqnarray}
|\langle Y\rangle_{n} -  \lambda^2  |    \leq c\,  \lambda^{2+\rho}\varepsilon_n^\rho      + \lambda^2\delta_n^2
\end{eqnarray}
(see  (\ref{sec877})
for a similar argument).
Thus,  $Y$ satisfies the following conditions. For all $0\leq \lambda =o(\max\{ \varepsilon_n^{-1},  \kappa_n^{-1}  \}),$
\begin{description}
\item[(C1)] $\sum_{i=1}^n\mathbf{E}_{\lambda}\big[|\eta_{i}/\lambda|^{2+ \rho} \,  \big | \mathcal{F}_{i-1} \big] \leq  c_\rho \,  \varepsilon_n^\rho \, $;
\item[(C2)] $ \left| \left\langle Y\right\rangle _n/\lambda^2- 1\right| \leq  c\,    \lambda^{ \rho}\varepsilon_n^\rho   +  \delta_n^2.$
\end{description}

In the proof we make use of the following  lemma  of  \citep{J93}.
\begin{lemma} Let $\rho >0$. Then
\begin{equation}
\sup_{ x  } \Big|\mathbf{P}(M_n \leq x)  -  \Phi \left( x\right) \Big|    \leq  C_p \bigg(  \Big(\sum_{i=1}^n\mathbf{E}|\xi_{i}| ^{2+\rho} \Big )^{1/(3+\rho) }  + \left|| \langle M \rangle_n-1\right ||^{1/2}_\infty   \bigg).
\end{equation}
\end{lemma}
  Applying the last lemma to the martingale $Y/\lambda,$ we obtain the desired inequalities.